\documentclass[]{amsart} 

\usepackage{amsmath,amsthm,amssymb,amsfonts,mathrsfs,color,hyperref, mathtools,crop, graphicx, enumitem, todonotes}
\usepackage[bottom=3cm, left=3cm, right=3cm]{geometry}
\usepackage[g]{esvect}

\theoremstyle{plain}
\begingroup
\newtheorem{theorem}{Theorem}[section]
\newtheorem*{theorem*}{Theorem}
\newtheorem*{"theorem"}{``Theorem''}
\newtheorem{corollary}[theorem]{Corollary}

\newtheorem{lemma}[theorem]{Lemma}
\endgroup

\theoremstyle{definition}
\begingroup

\endgroup

\theoremstyle{remark}
\begingroup
\newtheorem{remark}[theorem]{Remark}

\endgroup 

\numberwithin{equation}{section}
\newenvironment{pde}{\left\{\begin{array}{rll} } {\end{array}\right.}

\newcommand{\N}{\mathbb N} 
 
\newcommand{\Z}{\mathbb Z} 
\newcommand{\R}{\mathbb R} 

\newcommand{\supp}{\mathrm{supp}}

\renewcommand{\H}{{\mathcal H}}
\newcommand{\E}{{\mathcal E}}
\newcommand{\W}{{\mathcal W}}
\newcommand{\M}{{\mathcal M}}

\newcommand{\LRa} {\Leftrightarrow}
\newcommand{\Ra} {\Rightarrow}
\newcommand{\wto}{\rightharpoonup}

\renewcommand{\d}{\,\mathrm{d}}

\newcommand{\dx}{\,\mathrm{d}x}
\newcommand{\dy}{\,\mathrm{d}y}
\newcommand{\dz}{\,\mathrm{d}z}
\newcommand{\ds}{\,\mathrm{d}s}

\DeclareMathOperator*{\argmin}{arg\,min}

\newcommand{\cc}{\Subset}

\let \ol = \overline

\newcommand{\eps}{\varepsilon}
\newcommand{\average}{{\mathchoice {\kern1ex\vcenter{\hrule height.4pt
width 6pt depth0pt} \kern-9.7pt} {\kern1ex\vcenter{\hrule
height.4pt width 4.3pt depth0pt} \kern-7pt} {} {} }}

\newcommand{\com}[1]{\textcolor{black}{#1}}
\allowdisplaybreaks
 
\begin{document}

\title[Confined elasticae and cylindrical shells]{Confined elasticae and the buckling of cylindrical shells}

\author{Stephan Wojtowytsch}
\address{Stephan Wojtowytsch\\
Princeton University\\
PACM\\
205 Fine Hall - Washington Road\\
Princeton, NJ 08544
}
\email{stephanw@princeton.edu}

\date{\today}


\begin{abstract}
For curves of prescribed length embedded into the unit disc in two dimensions, we obtain scaling results for the minimal elastic energy as the length just exceeds $2\pi$ and in the large length limit. In the small excess length case, we prove convergence to a fourth order obstacle type problem with integral constraint on the real line which we then solve. From the solution, we \com{obtain the energy expansion $2\pi + \Theta \delta^{1/3} + o(\delta^{1/3})$ when a curve has length $2\pi + \delta$ and determine first order coefficient $\Theta\approx 37$}. We present an application of the scaling result to buckling in two-layer cylindrical shells where we can determine an explicit bifurcation point between compression and buckling in terms of universal constants and material parameters scaling with the thickness of the inner shell.
\end{abstract}

\maketitle

\tableofcontents

\section{Introduction}

For a curve $\gamma$ in two or three dimensions, we define the elastic energy
\[
\W(\gamma) = \int_\gamma \kappa^2 \d\H^1
\]
where $\kappa$ is the curvature of $\gamma$ and $\H^1$ is the one-dimensional Hausdorff measure. This energy is a simple model for elastic beams, the leading order elastic energy of (almost) straight thin sheets, and is proposed in image segmentation to reconstruct objects partially occluded from the viewer which a perimeter-based functional may not capture. Below, we will consider a specific application in two-layer cylindrical shells (such as tubes composed of two different materials).

Our notation is derived from the corresponding energy on surfaces, which is usually referred to as Willmore's energy. For the Willmore functional, M\"uller and R\"oger have considered the following problem: {\em Minimise $\W$ among all surfaces $\Sigma$ embedded in the unit ball $B_1(0)$ in three dimensions which have prescribed area $S>0$.} In \cite{Muller:2013vz}, they prove that 
\[
\limsup_{S\to\infty} \inf_{|\Sigma|=S} \left[\W(\Sigma) - S\right] < \infty, \qquad \inf_{|\Sigma|=S}\W(\Sigma) = S \quad\LRa\quad S\in 4\pi\Z
\]
and, perhaps most interestingly, that there exist constants $c,C>0$ and $\delta_0>0$ such that
\[
4\pi + c\delta^{1/2} \leq \inf_{|\Sigma| = 4\pi + \delta} \W(\Sigma) \leq 4\pi + C\,\delta^{1/2} \qquad \forall\ 0<\delta<\delta_0.
\]
In this article, we obtain the analogous results for curves in the plane. While the proof in \cite{Muller:2013vz} invokes rigidity estimates for nearly umbilical surfaces due to De Lellis and M\"uller, our arguments are elementary by comparison and we characterize the leading order term in the energy explicitly instead of just giving scaling bounds.

In analogy to \cite{Muller:2013vz}, we introduce the space of admissible curves
\[
\M_L = \left\{\gamma\in C^\infty\big(S^1; B_1(0)\big) \:\bigg|\: |\gamma'|\equiv \frac{L}{2\pi}, \:\gamma\text{ is embedded}\right\}
\]
for $L>0$ and consider the problem of minimizing $\W$ in $\M_L$. This can be thought of as a geometric higher order obstacle problem where the obstacle is given by the domain boundary $\partial B_1(0)$ and the curve itself due to the non-self intersection constraint. We have a technical advantage over the setting considered in \cite{Muller:2013vz} since curves unlike surfaces admit arc-length parametrisations and we can avoid the language of geometric measure theory entirely.
In the short \com{length} regime, we show the following.

\begin{theorem}\label{theorem short length}
\com{The energy $\W$ satisfies
\[
\inf_{\gamma \in \M_{2\pi + \delta}}
\W(\gamma) = 2\pi + \Theta\,\delta^{1/3} + o(\delta^{1/3})
\]
as $\delta\searrow 0$ where
\[
\Theta = \inf\left\{ \int_\R |\phi''|^2\dx\:\bigg|\:\phi \in C_c^\infty(\R),\quad\phi\geq 0, \quad \int_\R\frac{(\phi')^2}2 - \phi \dx = 1\right\}.
\]
}
\end{theorem}

So \com{when} the length of the curve just \com{$2\pi$} where it can fit into the domain $B_1(0)$ as a circle, $\W$ shows a steep increase in terms of the excess length. The qualitative behaviour is therefore comparable to that in the two-dimensional case while the order of the rapid growth is different. 

The key argument in proving the theorem is showing that it is asymptotically equivalent to a higher order obstacle type problem with an integral constraint on the real line which we obtain by a careful expansion procedure. 

Denote the linearised length \com{and} energy functional by 
\[
L, \E:C_c^\infty(\R)\to \R, \qquad L(\phi) = \int_{\R}\frac{(\phi')^2}2 - \phi \ds, \qquad \E(\phi) = \int_\R |\phi''|^2\ds.
\]
The non-linear space $\M_L$ is replaced by the manifold
\[
M= \left\{ \phi \in C_c^\infty(\R)\:\bigg|\:\phi\geq 0, \:\: \int_{\R}\frac{(\phi')^2}2 - \phi \ds =1\right\}.
\]

\begin{theorem}\label{theorem line}
The energy $\E$ has a minimiser $\ol u$ in the larger class
\[
\overline M = \left\{\phi \in W^{2,2}(\R)\cap L^1(\R)\:\bigg|\:\phi\geq 0, \:\: \int_{\R}\frac{(\phi')^2}2 - \phi \ds =1\right\}
\]
that satisfies the following properties:
\begin{itemize}
\item $1 < \E(\ol u) < \infty$.
\item $\ol u \in C^{2,1}_c(\R)\setminus C^3(\R)$.
\item $\ol u$ is compactly supported and the set $\{\ol u>0\}$ is connected.
\item $\ol u$ is even.
\item $\ol u$ increases from $0$ to its maximum in a monotone fashion.
\end{itemize}
\end{theorem}

There is an interesting scaling relation in the problem. Namely, for given $\phi$ the function $\phi_\rho(x) = \rho^{2/3}\,\phi\left(\rho^{-1/3} x\right)$ satisfies
\[
L(\phi_\rho) = \rho\cdot L(\phi), \qquad \E(\phi_\rho) = \rho^{1/3}\cdot \E(\phi)
\]
which explains the emergence of the problem on the line from the geometric problem: For small $\delta$, $L(\phi_\delta)$ approximates the excess length of a curve which is given by a slight perturbation of a circle, described by a $\phi_\delta$-shaped bump in radial direction. The condition $\phi\geq 0$ is required to ensure that the corresponding curve lies inside the circle. The parameter $\Theta:= \inf_{\phi\in M}\E(\phi)$ connects Theorems \ref{theorem short length} and \ref{theorem line} and the same connection together with the scaling property explains the emergence of the optimal order $1/3$.

Note that the regularity $W^{3,\infty}\setminus C^3$ is strictly higher than the optimal regularity in more classical obstacle problems. The minimiser is obtained as the limit of minimisers of similar problems over compact intervals and can be described fairly explicitly as the solution of an Euler-Lagrange equation, which in one dimension is just an ODE. This non-homogeneous linear fourth order ODE can be solved explicitly, and we can analyse it further to explicitly obtain the minimiser
\[
\ol u(x) = \begin{cases} a - \frac {x^2}2 + \alpha\cos(\mu x) &|x|<r\\ 0 &\text{else}\end{cases}
\]
with parameters $r=\sqrt[3]{6}\approx 1.82$ and $\mu\approx 2.47$ such that $\mu r$ is the first positive solution of the equation $\tan(\rho) = \rho$. The remaining parameters $a\approx1.81, \alpha\approx 0.75$ and an energy $\Theta\approx 36.69$ can be computed analytically from $\mu$ and $r$. We give a graphical representation below in Figure \ref{figure minimiser}.

Besides energy minimisation, we describe regularity properties of functions satisfying energy bounds and develop a general variational machinery for the problem. 
We also obtain analogous results for the minimisation problems of the energies
\[
\W_\alpha(\gamma) = \W(\gamma) + \alpha\,\H^1\big(\{\gamma\notin \partial B_1(0)\}\big), \qquad \E_\alpha(\phi) = \E(\phi) + \alpha\,\H^1\big(\{\phi\neq 0\}\big)
\]
in the respective classes $\M_{2\pi+\delta}$ and $M$. Many results carry over, except that minimisers of $\E_\alpha$ are only $C^{1,1}$-smooth and not $C^2$-smooth if $\alpha>0$. We apply the results of Theorem \ref{theorem short length} to a model of the following problem: Imagine two cylindrical shells, one contained in the other, of which the outer one has the same height but smaller area (say, a pipe made of two layers where the outer one contracts more at low temperatures). The inner layer has two options to comply with the constraint forced by the outer layer: compression or buckling. Assuming that all shells remain cylindrical and that the outer layer is a lot more rigid than the inner one, we show that bifurcation to buckling would be expected at
\[
\delta = \lambda_0^{-3/5}\, \left( \frac{\Theta c_{mat}}{r_o^{4/3}}\right)^\frac35\,h^{6/5}
\]
where $\delta$ is the excess preferred length of the planar profile of the inner shell, $h$ is the thickness of the inner shell, $r_o$ is the radius of the (circular) outer shell, $c_{mat}$ is a material constant, the universal constant $\Theta$ is as above and $\lambda_0 \approx 1.034$ is the parameter such that the function
\[
f(s) = (s-1)^2 + \lambda\,s^{1/3}
\]
has its minimum at $0$ if $\lambda>\lambda_0$ and at a positive point if $\lambda<\lambda_0$. If we include an adhesive between the shells in the model by considering a bending energy $\W_\alpha$ with $\alpha>0$ which penalises delamination, the buckling regime changes from $h^{6/5}$ to $(\alpha h)^{2/5}$.

We also characterise the large length limit. Since in general $\W(\gamma)\geq \H^1(\gamma)$ in analogy to the two-dimensional case \cite{Muller:2013vz}, we only need to show an estimate of the form $\W(\gamma) \leq L + c_L$ for curves of length $L\gg 1$. 

\begin{theorem}\label{theorem large length}
In the large length limit, we observe that
\[
\limsup_{L\to\infty} \frac{\inf_{\gamma \in \M_L}\W(\gamma) - L}{\sqrt{L}} < \infty.
\]
\end{theorem}

Finally, we demonstrate that the sharp energy increase is a two-dimensional phenomenon which can be avoided in three (or more) dimensions by out-of-plane buckling.

\begin{theorem}\label{theorem 3d}
Using the same notation as above, but assuming that $B_1(0)$ is the unit ball in three (or more) dimensions, there exists a constant $1<C\leq \frac 92$ and $\delta_0>0$ such that
\[
2\pi + \delta \leq \inf_{\gamma\in\M_{2\pi + \delta}} \W(\gamma) \leq 2\pi + C\delta \qquad \forall\ \delta>\delta_0.
\]
Furthermore, 
\[
\limsup_{L\to\infty} \left({\inf_{\gamma \in \M_L}\W(\gamma) - L}\right)<\infty.
\]
\end{theorem}

Theorems \ref{theorem large length} and \ref{theorem 3d} are proved by explicit constructions of energy competitors. The article is structured as follows. In Section \ref{section review}, we review some classical results on elastic curves and collect a few results on their variational structure. Section \ref{section short length} is devoted to the study of the minimisation problem for small excess length and the proof of Theorem \ref{theorem short length} where we derive the linearised obstacle problem on the line, which is then treated in Section \ref{section line}. In this section, we also consider the problem with a delamination penalty and applications to the buckling of elastic shells. Finally, in Sections \ref{section large length} and \ref{section 3d}, we discuss the large length limit and the situation for space curves. The sections are essentially independent and can be read separately, with the exception of Section \ref{section review}, which is needed for all but Section \ref{section line}. All results are discussed and put into context in Section \ref{conclusion}. In the appendix, we give quick proofs of some of the results from Section \ref{section review}.

\section{Review of Elastic Curves}\label{section review}

The energy $\W$ is geometric in nature, i.e.\ independent of the parametrisation of a curve. This allows an {\em extrinsic} approach through generalised objects in geometric measure theory known as varifolds. While this technicality is mostly unavoidable for higher dimensional versions of this problem, in one dimension, we have viable alternatives. In particular, if a curve is parametrised by unit speed, $|\gamma'|=1$, then its curvature vector is given by 
\[
\vv H_{\gamma(s)} = \gamma''(s)
\]
and its curvature is \com{$\kappa_{\gamma(s)} = 
\langle \nu,\vv H\rangle = \langle\gamma'', (\gamma')^\bot\rangle$ where $\nu$ denotes the unique normal vector to $\gamma$. The second expression holds if $\gamma$ is parameterized by arc length and $v^\bot$ is the rotation of $v$ by ninety degrees. The sign is chosen such that circles have positive curvature.} In particular, the elastic energy of a curve can be written as
\[
\W(\gamma) = \int_{S_L^1}|\gamma''|^2\ds
\]
if $\gamma$ is parametrised by arc-length on a circle of length $L$. This is a valuable tool in the calculus of variations, allowing us to relate the problem to the Sobolev space $W^{2,2}(S_L^1;\R^n)$. In a later chapter, we will use a radial parametrisation of curves, so we note that the curvature of a curve in general parametrisation is given by
\[
\vv H_{\gamma} = \frac{\gamma'' - \left\langle\gamma'', \frac{\gamma'}{|\gamma'|}\right\rangle \frac{\gamma'}{|\gamma'|}}{|\gamma'|^2}
\]
as easily confirmed by the chain rule. This makes the elastic energy
\begin{align*}
\W(\gamma) &= \int_{S_L^1} \left|\frac{\gamma'' - \left\langle\gamma'', \frac{\gamma'}{|\gamma'|}\right\rangle \frac{\gamma'}{|\gamma'|}}{|\gamma'|^2}\right|^2\,|\gamma'|\ds\\
	&= \int_{S_L^1} \frac{|\gamma''|^2 - 2\left\langle\gamma'', \frac{\gamma'}{|\gamma'|}\right\rangle\,\left\langle \frac{\gamma'}{|\gamma'|},  \gamma''\right\rangle +\left\langle\gamma'', \frac{\gamma'}{|\gamma'|}\right\rangle^2} {|\gamma'|^3}\ds\\
	&= \int_{S_L^1} \frac{|\gamma''|^2 - \left\langle \gamma'', \frac{\gamma'}{|\gamma'|}\right\rangle^2}{|\gamma'|^3}\ds.
\end{align*}
Due to the geometric nature of the energy, we will not distinguish between the trace of a curve in $\R^n$ and its parametrisations and reparametrisations. Similarly, we identify the circle $S^1_L$ of length $L$ with the periodic interval $[0,L]$ or $\R/L\Z$.

Let us review classical results for elastic curves in any dimension $\com{n}\geq 2$ which we prove for the reader's convenience in the appendix.

\begin{lemma}\label{lemma basics}
\begin{enumerate}
\item Let $\gamma$ be any curve and $\alpha\neq 0$. Then $\W(\alpha \gamma) = \frac{\W(\gamma)}{|\alpha|}$.
\item Let $L>0$ and $\gamma$ be a $W^{2,2}$-curve of length $L$. Then $\W(\gamma)\geq \frac{4\pi^2}{L}$ and equality holds if and only if $\gamma$ is a circle.
\item For any $W^{2,2}$-curve $\gamma$, energy and length are related by $\W(\gamma) \geq \frac{4\pi^2}{\H^1(\gamma)}$.
\item If there exists a point $x\in\gamma$ of multiplicity $k$, i.e.\ there exists $x\in \R^n$ such that 
\[
x = \gamma(t_1) = \dots= \gamma(t_k)
\]
for $k$ distinct parameters $t_1,\dots,t_k\in S^1$, then
\[
\W(\gamma)\geq \frac{C\,k^2}{\com{\H^1(\gamma)}}
\]
for a constant $C>\pi^2$.
\end{enumerate}
\end{lemma}

This means that the unique minimiser (up to Euclidean motion) of $\W$ among $W^{2,2}$-curves with given length is the once covered circle. Given the rescaling property $\W(\alpha\gamma) = \frac{\W(\gamma)}{|\alpha|}$, the functional $\W$ has no critical points since we can always reduce the energy of a curve by making it larger, corresponding to the variation in radial direction. However, there are critical points under a length constraint, or equivalently critical points of the scale-invariant functional
\[
\widetilde \W(\gamma) = \H^1(\gamma)\cdot \W(\gamma).
\]
The following result is deeper and characterises the critical points of $\widetilde\W$. These curves are often referred to as (Euler) elasticae. 

\begin{theorem}\label{theorem periodic elasticae}\cite[Theorem 1]{MR3096373}
Let $\gamma$ be a critical point of $\widetilde \W$. Then one of the three following holds.
\begin{enumerate}
\item $\gamma$ is a once or multiply covered circle.
\item $\gamma$ is a particular once covered figure eight curve.
\item $\gamma$ is a multiple cover of the same figure eight curve.
\end{enumerate}
In the first two cases, $\gamma$ is a stable critical point, in the third one, it is unstable.
\end{theorem}

There are other elasticae (critical points of $\widetilde \W$ under compact perturbations) which are not periodic; in fact, planar elasticae were classified into 9 different families already by Euler in 1744 \cite{levien2008elastica}.
Only one of the closed elasticae, the once covered circle, is approximable by embedded curves.

\begin{lemma}\label{lemma approximable curves}
Let $\gamma\in C^\infty(S^1;\R^2)$ be a closed elastica and $\gamma_n\in C^1(S^1;\R^2)$ a sequence of embedded curves such that $\gamma_n\to \gamma$ in $C^1$. Then $\gamma$ is the once covered circle.
\end{lemma}

A proof can be found in the appendix.

Let us consider the problem of minimising the elastic energy $\W$ in the class of curves which are embedded into a domain $\Omega\subset \R^2$ with prescribed length. Some properties of the problem which were originally proved in \cite{MR2097034, dondl:2011eh} are described in the following Lemma.

\begin{lemma}\label{lemma belletini-mugnai}
Let $\Omega \cc\R^2$ be an open set with Lipschitz boundary and $L>0$. We set
\[
\M_L^{\com{\Omega}} = \left\{\gamma\in C^\infty\big(S^1; \Omega \big) \:\bigg|\: |\gamma'|\equiv \frac{L}{2\pi}, \:\gamma\text{ is embedded.}\right\}.
\]
Denote by $\ol \M_L^\Omega$ the closure of $\M_L^\Omega$ in the $W^{2,2}$-weak topology. The following are true.
\begin{enumerate}
\item $\ol \M_L^\Omega$ coincides with the closure of $\M_L^\Omega$ in the $W^{2,2}$-strong topology.
\item There exists a minimiser $\ol \gamma$ of $\W$ in $\ol\M_L^\Omega$.
\item $\inf_{\gamma \in \M_L^\Omega}\W(\gamma) = \min_{\gamma\in \ol\M_L^\Omega}\W(\gamma)$.
\item The minimiser \com{satisfies at least one of the following: It }is  a circle, touches the boundary, or has at least one multiple point.
\end{enumerate}
\end{lemma}

Essentially, the theorem shows that a curve which arises as the weak limit of smooth embedded curves with prescribed length and bounded energy can also be approximated strongly in the $W^{2,2}$-topology. This property is well-known for convex sets in Banach-spaces, but the embeddedness constraint is highly non-convex.
At minimisers, it shows that the problem does not exhibit the Levrentiev-gap phenomenon where smoothness is incompatible with low energy.

\com{For the unit disk $\Omega = B_1(0)$, we write $\M_L = \M_L^{B_1(0)}$.}

\section{Minimisation Problem at $2\pi$}\label{section short length}

\com{From now on, we will only consider $\Omega = B_1(0)$.}

\begin{lemma}\label{lemma properties minimisers 2pi}
Let $\delta_n \com{\searrow} 0$ and $\gamma_n\in \overline{\M}_{2\pi + \delta_n}$ be a sequence such that $\W(\gamma_n) = \inf_{\gamma\in \M_{2\pi + \delta_n}}\W(\gamma)$. Then
\begin{enumerate}
\item there exists $C>0$ such that $\W(\gamma_n) \leq 2\pi + C\delta_n^{1/3}$,
\item $\gamma_n$ converges to the unit circle strongly in $W^{2,2}(S^1;\R^2)$ (up to reparametrisation) and
\item for \com{all sufficiently large} $n$ there exists a parameter $s\in S^1$ such that $|\gamma_n(s)|=1$, i.e.\ $\gamma_n\not\subset B_1(0)$.
\end{enumerate}
\end{lemma}

\begin{proof}
{\bf Energy bound.} Clearly, it suffices to construct curve $\tilde\gamma_n\in \overline\M_{2\pi + \delta_n}$ such that $\W(\tilde\gamma_n)\leq 2\pi + C\,\delta_n^{1/3}$. This is by far the longest part of the proof and concluded in Lemmas \ref{lemma length estimate} and \ref{lemma upper energy bound}.

{\bf Convergence to the unit circle.} By compactness, up to a subsequence, we see that there exists a curve $\gamma\in \overline\M_{2\pi}$ such that $\gamma_n\wto \gamma$ weakly in $W^{2,2}$ and thus strongly in $C^1$ -- in particular, $\gamma$ is parametrised by arc-length and has length $2\pi$. Furthermore, 
\[
\W(\gamma)\leq \liminf_{n\to\infty}\W(\gamma_n) = 2\pi.
\]
As a consequence, $\gamma$ is a curve of length $2\pi$ and energy $2\pi$, which can only be realised by a circle \com{(see part (2) of Lemma \ref{lemma basics}).}

Since $\W(\gamma_n)\to \W(\gamma)$, by a common Hilbert space argument we find that $\gamma_n\to \gamma$ strongly in $W^{2,2}$. Furthermore, since the limiting object is the unique circle $\partial B_1(0)$ of length $2\pi$ in $\overline{B_1(0)}$, we can (after reparametrising the curves if necessary) show that the whole sequence converges.

{\bf Touching the circle.} Assume that $\gamma_n\subset B_1(0)$. Since we assumed $\gamma_n$ to be a minimiser, $\gamma_n$ must be a critical point of $\W$ under the length constraint {\em without} the confinement constraint. Since $\W(\gamma_n) \leq 2\pi + C\delta_n^{1/3}$, we can see from Statement 4 in Lemma \ref{lemma basics} that $\gamma_n$ has no double points. \com{In summary} this means that $\gamma_n$ is embedded in $B_1(0)$ and we can take variations of $\gamma_n$ in all directions. But then, due to Theorem \ref{theorem periodic elasticae} $\gamma_n$ must be a once or multiply covered circle or figure eight. Both the multiply covered circle and any cover of the figure eight are ruled out in Lemma \ref{lemma approximable curves} (or by the fact that they have double points) so $\gamma_n$ has to be a once covered circle. However, $\gamma_n\subset B_1(0)$ and $\H^1(\gamma_n)>2\pi$ which means that $\gamma_n$ cannot be a circle. We have reached a contradiction.
\end{proof}

Since $\gamma_n$ is $C^1$-close to a circle for $n$ large enough, we can write it as a normal graph over the unit circle, i.e.\ there exists a function
\[
\phi_n\in W^{2,2}\big(S^1; [0,1/2]\big)\qquad\text{such that}\quad \gamma_n(s) = \big(1-\phi_n(s)\big)\begin{pmatrix} \cos s\\ \sin s\end{pmatrix}
\]
up to reparametrisation. The energy can now be re-written in terms of $\phi_n$. In the following, we will drop the subscript $n$ and simply write $\gamma, \phi,\delta$ instead of $\gamma_n, \phi_n, \delta_n$. When varying $\delta$, we may make the dependence explicit by writing $\phi_\delta, \gamma_\delta$.

We now compute arc-length element, curvature, length and energy of $\gamma$ in terms of $\phi$ for general curves presented in radial form.
 \begin{align*}
\gamma' &= (1-\phi) \begin{pmatrix} -\sin s\\ \cos s\end{pmatrix} - \phi'\begin{pmatrix} \cos s\\ \sin s\end{pmatrix}\\
|\gamma'| &= \sqrt{(1-\phi)^2 + (\phi')^2}\\
\gamma'' &= (1-\phi)\,\begin{pmatrix}-\cos s\\ -\sin s\end{pmatrix} - 2\phi' \begin{pmatrix} - \sin s\\ \cos s\end{pmatrix} -\phi'' \begin{pmatrix}\cos s\\ \sin s\end{pmatrix}\\
	&= (\phi -1 - \phi'')\,\begin{pmatrix} \cos s\\ \sin s\end{pmatrix} -2\phi' \begin{pmatrix} -\sin s\\ \cos s\end{pmatrix}
\end{align*}
\com{In particular, we see that $\phi, \phi', \phi''$ are all $L^2$-close to zero since $\gamma$ is $H^2$-close to being the unit circle. We recall the following estimate for almost circular curves.
\begin{lemma}\cite[Lemma 2.1]{goldman2018quantitative}\label{lemma new reference}
\[
\W(\gamma) = 2\pi + \int_{S^1} (\phi'')^2 + \phi^2 + \frac32 (\phi')^2 +\phi + 4\phi \,\phi''\ds + o\big(\|\phi\|_{H^2}^2\big).
\]
\end{lemma}
This can be refined as follows.
\begin{corollary}\label{corollary leading order expansion}
Let $\gamma:S^1\to \overline{B_1(0)}$ be given in radial coordinates 
\[
\gamma(s) = \big(1-\phi(s)\big)\,\begin{pmatrix}\cos(s)\\ \sin(s)\end{pmatrix}
\]
and $\H^1(\gamma) = 2\pi + \delta$ for $\delta\ll 1$. Then 
\[
\W(\gamma) = 2\pi + \|\phi''\|_{L^2(S^1)}^2 + o\big(\|\phi''\|_{L^2(S^1)}^2\big).
\]
\end{corollary}
\begin{proof}
Note that by standard estimates
\[
\|\phi\|_{L^1}\leq \sqrt{2\pi} \|\phi\|_{L^2} \leq C \|\phi'\|_{L^2}.
\]
It thus suffices to show that $\phi'$ is small in $L^2$ compared to $\phi''$. By the concavity of the square-root function, we have
\[
2\pi+\delta = \H^1(\gamma) = \int_{S^1} \sqrt{(1-\phi)^2 + (\phi')^2}\ds \leq \int_{S^1} 1 + \frac{(\phi')^2}2 - \phi + \frac{\phi^2}2\ds
\]
which implies the non-linear relation
\[
\int \phi \ds + \delta\leq \frac12 \int (\phi')^2 + \phi^2\ds \leq C\int (\phi')^2\ds.
\]
We can further deduce that
\begin{align*}
\|\phi'\|_{L^2}^2 &= - \int_{S^1}\phi \phi''\ds\\
	&\leq \|\phi\|_{L^2} \|\phi''\|_{L^2}\\
	&\leq C \,\|\phi\|_{L^1}^{1/2}\|\phi\|_{L^\infty}^{1/2}\,\|\phi''\|_{L^2}\\
	&\leq C \,\|\phi'\|_{L^2}^{3/2}\,\|\phi'\|_{L^2}
\end{align*}
by the non-linear estimate for the $L^1$-norm and the standard embedding for the $L^\infty$-norm. Thus we find that
\[
\|\phi'\|_{L^2} \leq C\,\|\phi''\|_{L^2}^2
\]
which establishes the estimate for almost all terms from Lemma \ref{lemma new reference}. The integral of $\phi\phi''$ can be converted into the integral of $(\phi')^2$ by integration by parts, concluding the proof.
\end{proof}
}
Before arguing for a general function $\phi$ which is $C^1$-close to $0$, let us make a specific ansatz $\psi_\delta(s) = \delta^\alpha\psi\left(\delta^{-\beta}x\right)$ for some non-negative function $\psi \in C_c^\infty(\R)$ which, for small enough $\delta$, induces a function $\psi_\delta\in W^{2,2}(S^1)$, once rescaled sufficiently to force the support into an interval of length $2\pi$. We calculate
\begin{align*}
2\pi + \delta &\stackrel! = \H^1(\gamma^{\psi_\delta})\\
	&= \int_{S^1} \sqrt{(1- \delta^\alpha\psi)^2 + (\delta^{\alpha-\beta}\psi')^2}\,(\delta^{-\beta}s)\ds\\
	&= \int_{S^1} 1 + \frac12 \left(-2\delta^{\alpha}\psi + (\delta^\alpha\psi)^2 + (\delta^{\alpha-\beta}\psi')^2\right)(\delta^{-\beta}s)\\
	&\hspace{3cm} + O\left(\big(-2\delta^{\alpha}\psi + \delta^{2\alpha}\psi^2 + \delta^{2\alpha-2\beta}(\psi')^2\big)^2\right)(\delta^{-\beta} s)\ds\\
	&= 2\pi + \int_{\R} \delta^{2\alpha-\beta}\frac{(\psi')^2(x)}2 - \delta^{\alpha+\beta}\psi(x) + \delta^{2\alpha+\beta}\frac{\psi^2(x)}2\dx + O\left(\delta^{2\alpha} + \delta^{2(2\alpha-\beta)}\right)\cdot\delta^\beta
\end{align*}
where the error term is multiplied by the length of the support of the functions.
We see that the dominant term must be $O(\delta)$, and since the second term is always negative as $\psi\geq 0$ and always dominates the third term, we see that
\[
\alpha+\beta \geq 2\alpha - \beta = 1\qquad \Ra\qquad \beta = 2\alpha -1, \quad \beta\geq \frac\alpha 2.
\]
If $\beta > \frac{\alpha}2$, a function with $\int_{\R}\frac{(\psi')^2}2\dx = 1$ will match up to leading order and if $\beta=2\alpha$, a function satisfying $\int_{\R} \frac{(\psi')^2}2-\psi\dx=1$ matches the right length up to leading order. In order to \com{have a low energy $\W$}, we want to maximise $\alpha$ which means
\[
\frac\alpha2 = \beta = 2\alpha-1\qquad \Ra\qquad \alpha = \frac23, \:\beta = \frac13.
\]
For any $\rho>0$, we therefore set $\psi_\rho(x) = \rho^{2/3}\psi(\rho^{-1/3}x)$. We now see that the error term is 
\[
\com{O\left(\delta^{2\alpha+\beta}\right)} + O\left(\delta^{2\alpha} + \delta^{2(2\alpha-\beta)}\right)\cdot\delta^\beta = O(\delta^{\com{5/3}}) + O\left(\delta^{4/3}\right) \cdot \delta^{1/3} = O(\delta^{5/3}).
\]
Since the length of $\gamma^{\psi_\rho}(s) := \big(1- \psi_\rho(s)\big) (\cos s, \sin s)$ depends continuously on the scaling parameter $\rho$, we \com{invoke the intermediate value theorem to prove} the following.

\begin{lemma}\label{lemma length estimate}
Let $\psi\in C_c^\infty(\R)$ be a function such that $\int_{\R}\frac{(\psi')^2}2 - \psi \dx = 1$. Then there exists a $\delta_0>0$ such that for all $\delta<\delta_0$ there exists $\rho(\delta)>0$ such that
\[
\H^1(\gamma^{\psi_{\rho(\delta)}}) = 2\pi + \delta, \qquad \limsup_{\delta\to 0} \frac{\left|\rho(\delta)-\delta\right|}{\delta^{5/3}} <\infty.
\]
\end{lemma}

Using this construction we bound the infimum energy from above. This also concludes the proof of Lemma \ref{lemma properties minimisers 2pi}. Let us expose a simple scaling relationship in the minimisation problem.

\begin{lemma}\label{lemma scaling}
Let $\phi \in W^{2,2}(\R)\cap L^1(\R)$ and define $\phi_\rho(x) = \rho^{2/3}\,\phi\left(\rho^{-1/3} x\right)$ for $\rho\neq 0$. Then
\[
\int_\R \frac{(\phi_\rho')^2}2-\phi_\rho \dx = \rho\,\int_\R \frac{(\phi')^2}2 - \phi \dx, \qquad \int_\R |\phi_\rho''|^2\dx= \rho^{1/3}\int_\R|\phi''|^2\dx
\]
\end{lemma}

\begin{proof}
A simple change of variables shows that 
\begin{align*}
\int_\R \frac{(\phi_\rho')^2}2-\phi_\rho \dx
	&= \int_\R \left(\frac{\left(\rho^{2/3}\rho^{-1/3}\phi'\right)^2}2 - \rho^{2/3}\phi\right)\left(\rho^{-1/3}x\right)\dx\\
	&= \int_\R \rho^{2/3}\left(\frac{(\phi')^2}2 - \phi\right)(\rho^{-1/3}x)\dx\\
	&= \rho \int_\R\frac{(\phi')^2}2-\phi\dz\\
\int_\R (\phi_\rho'')^2\dx &= \int_\R \left(\rho^{2/3}\rho^{-2/3}\phi''\right)^2(\rho^{-1/3}x)\dx\\
	&= \rho^{1/3}\int_\R\big(\phi''\big)^2 \dz.
\end{align*}
\end{proof}

\begin{lemma}\label{lemma upper energy bound}
Under the same assumptions as in Lemma \com{\ref{lemma length estimate}}, we have
\[
\W(\gamma^{\psi_{\rho(\delta)}}) = 2\pi + \delta^{1/3}\int_{\R} |\psi''|^2(s)\ds + \com{o(\delta^{1/3})}.
\]
\end{lemma}

\begin{proof}
\com{Follows directly from Corollary \ref{corollary leading order expansion} and Lemma \ref{lemma scaling}.}
\end{proof}

In the following, we will show that the order $\delta^{1/3}$ is in fact optimal. 

\begin{lemma}
There exists a $\delta_0>0$ such that for all $0<\delta<\delta_0$ the following holds: Let $\gamma$ be the minimiser of $\W$ in $\ol \M_{2\pi+\delta}$. Then $\W(\gamma) = 2\pi + \Theta\,\delta^{1/3} + o(\delta^{1/3})$ where
\[
\Theta:= \inf\left\{\int_\R|\phi''|^2\dx \:\bigg|\: \phi \in C_c^\infty(\R), \:\phi\geq 0, \quad\int_\R \frac{(\phi')^2}2 - \phi\dx = 1\right\}.
\]
\end{lemma}

In the following section we will prove that $\Theta>0$ and further properties of the minimisation problem for $\phi$ on the line.

\begin{proof}
{\bf Upper bound.} We have already seen that curves of the form $\gamma^{\psi_{\rho(\delta)}}\in \overline{\M_{2\pi+\delta}}$ satisfy
\[
\lim_{\delta\to 0} \frac{\W(\gamma^{\psi_{\rho(\delta)}}) - 2\pi}{\delta^{1/3}} = \int_\R |\phi''|^2\ds.
\]
For any $\eps>0$ we can choose the function $\phi$ in the admissible class such that $\int_\R|\phi''|^2\ds \leq \Theta+\eps$, so we can in particular construct a family of curves satisfying 
\[
\limsup_{\delta\to 0} \frac{\W(\gamma^{\psi_{\rho(\delta)}}) - 2\pi}{\delta^{1/3}} \leq \Theta+\eps.
\]
A diagonal sequence argument shows that 
\[
\limsup_{\delta\to 0} \frac{\inf_{\gamma\in \M_{2\pi+\delta}} \W(\gamma) - 2\pi}{\delta^{1/3}} \leq \Theta.
\] 

{\bf Lower bound.}
\com{
Since $\gamma$ is $W^{2,2}$-close to the circle, we can write $\gamma$ in radial coordinates as before with $\phi \in W^{2,2}(S^1)$ which is $W^{2,2}$-close to zero. Following the proof of Corollary \ref{corollary leading order expansion}, we see that
\[
\delta \leq \int_{S^1} \frac{(\phi')^2}2 - \phi + \frac{\phi^2}2\ds \leq \int_{S^1} \frac{(\phi')^2}2 - (1-\|\phi\|_\infty) \phi \ds.
\]
Also in the proof of Corollary \ref{corollary leading order expansion}, we see that
\[
\|\phi\|_{L^\infty} \leq C\|\phi'\|_{L^2} \leq C\|\phi''\|_{L^2}^2 \leq C\,\delta^{1/3}
\]
since we already know that $\|\phi''\|_{L^2}^2\leq C\delta^{1/3}$, i.e.
\[
\delta \leq\int_{S^1} \frac{(\phi')^2}2 - (1-C\delta^{1/3}) \phi \ds.
\]
Since $\gamma$ touches the circle, there exists $s_0\in S^1$ such that $\phi(s_0) = \phi'(s_0) = 0$. We can therefore consider $\phi \in W^{2,2}_0(-\pi,\pi) \subset W^{2,2}(\R)$. Anticipating the results of Corollary \ref{corollary rescaled problem}, we conclude that
\begin{align*}
\int_\R (\phi'')^2\ds &\geq\inf \big(1-C\delta^{1/3}\big)^{4/3}\,\delta^{1/3}\left\{\int_{S^1}|\phi''|^2\ds\:\bigg|\: \phi\in W^{2,2}(\R)\cap C_c^1(\R), \:\: \phi\geq 0, \:\: \int_{S^1}\frac{(\phi')^2}2 - \phi\ds =1\right\}\\
&\geq \big(1-C\delta^{1/3}\big)^{4/3}\:\Theta\,\delta^{1/3}.
\end{align*}
We conclude that
\[
\W(\gamma) \geq 2\pi + \Theta\,\delta^{1/3} + o(\delta^{1/3})
\]
for any curve $\gamma$ in the admissible class.
}
\end{proof}

\begin{corollary}\label{corollary rescaled problem}
Let $0<\eps<1$. Then we have 
\begin{align*}
\inf&\left\{\frac{\int_{S^1}|\phi''|^2\ds}{\delta^{1/3}}\:\bigg|\: \phi\in W^{2,2}(\R)\cap C_c^1(\R), \:\: \phi\geq 0, \:\: \int_{S^1}\frac{(\phi')^2}2 - (1-\eps)\phi\ds \geq \delta\right\}\\
	&\qquad = (1-\eps)^{\frac 43}\com{\delta^{1/3}}\:\inf \left\{\int_{S^1}|\phi''|^2\ds\:\bigg|\: \phi\in W^{2,2}(\R)\cap C_c^1(\R), \:\: \phi\geq 0, \:\: \int_{S^1}\frac{(\phi')^2}2 - \phi\ds =1\right\}.
\end{align*}
The same identity holds if we consider $L^1$ instead of $C_c^1$.
\end{corollary}

\begin{proof}
Take any $\phi \in W^{2,2}(\R)\cap L^1(\R)$ such that $\phi\geq 0$ and $\int_{\R}\frac{(\phi')^2}2 - (1-\eps)\phi\ds \geq \delta$. Introduce $\psi = \frac \phi{1-\eps}$ and observe that
\begin{align*}
\int_\R \frac{(\psi')^2}2 - \psi \ds &= \int_\R \frac{1}{(1-\eps)^2}\,\frac{(\phi')^2}2 - \frac{\phi}{1-\eps}\ds\\
	&= \frac1{(1-\eps)^2} \int_\R \frac{(\phi')^2}2 - (1-\eps)\phi\ds\geq \frac{\delta}{(1-\eps)^{\com2}}\\
\int_\R (\psi'')^2 &= \frac{1}{(1-\eps)^2} \int_\R(\phi'')^2\ds.
\end{align*}
Denote 
\[
\rho:= \left(\frac{1}{(1-\eps)^2}\int_{\R}\frac{(\phi')^2}2 - (1-\eps)\phi\ds\right)^{-1} \leq \frac{(1-\eps)^2}\delta.
\]
Then the rescaled function $\psi_\rho$ satisfies
\[
\int_{\R}\frac{(\psi_\rho')^2}2 - \psi_\rho\ds = \rho \int_{\R}\frac{(\psi')^2}2 -\psi \ds=1
\]
and
\[
\int_{\R}(\psi_\rho'')^2\ds = \rho^{1/3}\int_\R(\psi'')^2\ds = \frac{\rho^{1/3}}{(1-\eps)^2}\int_\R(\phi'')^2\ds \leq \frac{(1-\eps)^{2/3}}{\delta^{1/3}(1-\eps)^2}\int_\R(\phi'')^2\ds
\]
with equality if $\int_{\R}\frac{(\phi')^2}2 - (1-\eps)\phi\ds = \delta$. Since the process is entirely reversible (assuming we rescaled to length $\geq 1$ instead of $=1$), we have proved equality.
\end{proof}

We will show in the next section that the infimum is in fact positive, which shows that we have successfully identified the first order expansion of the minimal elastic energy with small excess length. 
We have thus proved the first of our main results, Theorem \ref{theorem short length}:
\[
\lim_{\delta\to 0}\inf_{\gamma \in \M_{2\pi+\delta}} \frac{\W(\gamma) - 2\pi}{\delta^{1/3}} = \inf\left\{\phi \in W^{2,2}(\R)\cap C_c^1(\R)\:\bigg|\: \phi\geq 0, \quad \int_\R \frac{(\phi')^2}2 - \phi \dx = 1\right\}.
\]

\begin{remark}
Note that the sequence in the proof of the upper bound satisfies 
\begin{align*}
\|\phi_\delta\|_{L^1} &= O(\delta), &\|\phi_\delta\|_{L^2}^2 &= O (\delta^{4/3}), & \|\phi'_\delta \|_{L^2}^2 &= O(\delta),&\|\phi''_\delta\|_{L^2} &= O(\delta^{1/3})\\
&&\|\phi_\delta\|_{L^\infty} &= O (\delta^{2/3}), & \|\phi'_\delta \|_{L^\infty} &= O(\delta^{1/3}),&\|\phi''_\delta\|_{L^{\com{\infty}}} &= O(1)
\end{align*}
while we only establish the sub-optimal orders
\begin{align*}
\|\phi_\delta\|_{L^1} &= O(\delta^{2/3}), &\|\phi_\delta\|_{L^2}^2 &= O (\delta), & \|\phi'_\delta \|_{L^2}^2 &= O(\delta^{2/3}),&\|\phi''_\delta\|_{L^2} &= O(\delta^{1/3})
\end{align*}
in the proof of the lower bound.
\end{remark}

\begin{remark}
We have seen that the energy minimiser $\gamma$ of length $2\pi + \delta$ is $W^{2,2}$-close to a circle (when parametrised in radial fashion). However, note that the set bounded by $\gamma$ is contained in the unit disk and has a boundary of length $\H^1(\gamma)>2\pi$, which means that the set cannot be convex. In particular, it has negative curvature at a point, and $\gamma$ is not $C^2$-close to the unit circle.
\end{remark}

\begin{remark}
Assume that $\gamma\subset B_R(0)$ and $\H^1(\gamma) = 2\pi R + \delta$ for some small $\delta>0$. Then we consider the curve $\gamma_R = \frac{\gamma}R \subset B_1(0)$ of length $\H^1(\gamma_R) = 2\pi + \frac{\delta}R$. The result above shows that $\W(\gamma_R) \geq 2\pi + \Theta\left(\frac \delta R\right)^{1/3}$ (to leading order) and thus
\[
\W(\gamma) = \frac1R \W(\gamma_R)\quad \geq \frac{2\pi}R + \frac{\Theta}{R^{4/3}}\,\delta^{1/3}
\]
(again to leading order) which means that the qualitative behaviour remains unchanged, but the prefactor $\Theta\,R^{-4/3}$ decreases quickly with increasing radius/decreasing boundary curvature. In domains with non-constant boundary curvature, we would therefore expect buckling of $\gamma$ at the points of {\em lowest} boundary curvature, at least if the domain is similar enough to a circle (e.g. an ellipse with very similar major and minor axes).
\end{remark}

\section{Minimisation Problem on the Line}\label{section line}

From the geometric problem of the previous section, in the asymptotic regime we inherit the minimisation problem of a boring energy over an interesting domain. We show that the competition between the terms $(\phi')^2$ -- which needs to be large in an integrated sense to compete with $\phi$, but cannot have large localized energy since $(\phi'')^2$ needs to be small -- leads to non-trivial behaviour.

Define the non-linear domain
\[
M= \left\{ \phi \in C_c^\infty(\R)\:\bigg|\:\phi\geq 0, \:\: \int_{\R}\frac{(\phi')^2}2 - \phi \ds =1\right\},
\]
its closure
\[
\ol M = \left\{ \phi \in W^{2,2}(\R)\cap L^1(\R)\:\bigg|\:\phi\geq 0, \:\: \int_{\R}\frac{(\phi')^2}2 - \phi \ds =1\right\},
\]
and the energy function 
\[
\E(\phi) = \int_\R |\phi''|^2\ds.
\]
Recall that we denote
\[
\Theta = \inf_{\phi\in M} \E(\phi) = \inf_{\phi\in \ol M} \E(\phi).
\]

\subsection{Existence of Minimisers}

In this section, we establish that energy minimisers exist and find them explicitly, together with the constant $\Theta$ from the previous section. Since it suffices show that $\Theta>0$ to obtain the energy scaling result from the previous section, we note that this section can be skipped by a reader only interested in the correct order of scaling. A much shorter proof that $\Theta>0$ is given below in Lemma \ref{lemma basic variational} and Corollary \ref{corollary theta>0}. \com{For this section, we assume that $\Theta>0$, which is proved below in Section \ref{section low energy}.}

We begin by showing that it is energetically favourable to create a single bump rather than many small bumps, assuming that $\Theta>0$.

\begin{lemma}\label{lemma mass confinement}
Assume that $\phi \in C_c^\infty(\R)$ is a function such that $\phi \geq 0$ and $\phi = \phi_1 + \phi_2$ where $\phi_1\phi_2\equiv 0$ and
\[
\int_\R \frac{(\phi_1')^2}2 - \phi_1 \dx = t, \qquad \int_\R \frac{(\phi_{\com 2}')^2}2 - \phi_{\com 2} \dx\ \com{ =1-t}.
\]
Then
\[
\E(\phi) \geq \Theta\left[ t^{1/3} + (1-t)^{1/3}\right] >\Theta
\]
if $t\in (0,1)$ and 
\[
\E(\phi) \geq \Theta \max\{t^{1/3}, (1-t)^{1/3}\}
\]
else. $\E(\phi)>\Theta$ unless $\phi_1\equiv 0$ or $\phi_2\equiv 0$.
\end{lemma}

\begin{proof}
If $t\in (0,1)$, then by the rescaling property of Lemma \ref{lemma scaling}, we have 
\[
\E(\phi) = \E(\phi_1) + \E(\phi_2) \geq t^{1/3}\Theta + (1-t)^{1/3}\Theta
\]
since $\phi_1$ and $\phi_2$ cannot be simultaneously non-zero and thus do not interact. If $t<0$, then $1-t>\com{1}$ and thus
\[
\E(\phi) \geq \E(\phi_2) \geq (1-t)^{1/3}\,\Theta
\]
and similarly if $t>1$.
\end{proof}

We deduce a few further a priori properties.

\begin{lemma}\label{lemma minimiser conditions}
Assume that there exists a function $\ol u\in W^{2,2}(\R)\cap L^1(\R)$ such that
\[
\com{\bar u \geq 0,\qquad}\int_\R \frac{(\ol u')^2}2 - \ol u\dx = 1,\qquad \E(\ol u) = \Theta.
\]
Then
\begin{enumerate}
\item The set $\{\ol u>0\}$ is connected and
\item $\ol u\in C^\infty(\{\ol u>0\})$ solves the Euler-Lagrange equation
\[
\ol u^{(4)} + \frac{\Theta}6\left(\ol u'' + 1\right) = 0
\]
on the set $\{\ol u>0\}$.
\end{enumerate}
\end{lemma}

\begin{proof}
The first statement follows directly from Lemma \ref{lemma mass confinement}. \com{Note that all functions we consider are $W^{2,2}$-smooth, thus also $C^{1,1/2}$-smooth. It is not hard to see that Lemma \ref{lemma mass confinement} remains intact in this setting.} 

For the second statement, we observe that $\ol u$ is a critical point of the scale-invariant functional $\frac{\E(\phi)}{L(\phi)^{1/3}}$ under the constraint $\ol u\geq 0$. On the set $\{\ol u>0\}$ -- which is open due to regularity -- we can take a variation in any direction $\psi \in C_c^\infty(\{\ol u>0\})$, leading to
\begin{align*}
0 &= \delta\left(\frac{\E}{L^{1/3}}\right)(\ol u;\psi)\\
	&= \frac{\delta \E(\ol u; \psi)}{L(\com{\ol u})^{1/3}} + \E(\ol u) \cdot \frac{-1}3 \,L(\ol u)^{-4/3}\,\delta L(\ol u;\psi)\\
	&= \frac1{L(\com{\ol u})^{1/3}} \left(\delta \E(\ol u; \psi) - \frac{\E(\ol u)}{3\,L(\ol u)}\:\delta L(\ol u;\psi)\right)\\
	&= \delta \E(\ol u; \psi) - \frac{\Theta}{3}\,\delta L(\ol u;\psi)\\
	&= \int_\R2\,\com{\ol u}''\psi'' - \frac{\Theta}3\left(\com{\ol u'}\,\psi' - \psi\right)\ds.
\end{align*}
\com{This is the weak formulation of a linear fourth order elliptic equation with constant coefficients in one dimension. Standard PDE theory implies the existence, uniqueness, and $C^\infty$-regularity of a solution on any connected component of $\{\ol u>0\}$ given boundary values. We can therefore integrate by parts and obtain}
\begin{align*}
	\int_\R\left(\ol u^{(4)} + \frac\Theta6(\com{\ol u}'' + 1)\right)\psi\ds = 0.
\end{align*}
\end{proof}

Now, we are ready to establish the existence of a minimiser and give an explicit characterisation. The following arguments are fairly direct and classical.

\begin{lemma}\label{lemma minimiser}
There exists a function $\ol u \in C^{2,1}_c(\R)$ such that
\[
\int_\R \frac{(\ol u')^2}2 - \ol u\dx = 1,\qquad \E(\ol u) = \inf_{\phi\in M} \E(\phi)\ \com{ = \Theta}.
\]
The function is $\ol u$ is even, compactly supported, not $C^3$-smooth on $\R$ and given by
\[
\ol u(x) = \begin{cases}  a - \frac{x^2}2 + \alpha\,\cos(\mu x)&x\in(-r,r)\\ 0&\text{else}\end{cases}
\]
for parameters $a \approx1.81 ,\alpha\approx 0.75 ,\mu \approx 2.47 ,r\approx 1.82$. Its energy is $\Theta\approx 36.69$.
\end{lemma}

\begin{proof}
{\bf Set-up.} For $R>0$, consider 
\[
\ol M_R= \left\{ \phi \in W^{2,2}_0(-R,R)\:\bigg|\:\phi\geq 0, \:\: \int_{-R}^R\frac{(\phi')^2}2 - \phi \ds =1\right\}.
\]
Then, by the direct method of the calculus of variations, there exists a minimiser $\phi_R\in \com{\ol M_R}$ of $\E$ where by an abuse of notation we denote
\[
\E: \ol M_R \to \R, \qquad \E(\phi) = \frac12 \int_{-R}^R|\phi''|^2\ds
\]
as the same functional on the new domain (with a normalising factor of $1/2$ included here for convenience). On the set $\{\phi_R>0\}$, we may take variations of $\phi$ in any direction and see that the Euler-Lagrange equation
\[
\int_{\{\phi_R>0\}} \phi''\,\psi'' + \lambda \left(\phi'\,\psi' - \psi\right) \ds = 0
\]
holds where $\lambda$ is a Lagrange multiplier stemming from the constraint and $\psi \in W^{2,2}_0(-R,R)$. Since \com{$\{\phi_R>0\}$} may coincide with \com{$(-R,R)$}, we \com{cannot conclude that $\phi_R$ is a critical point of $L^{-1/3}\E$ as we did for $\ol u$ in Lemma \ref{lemma minimiser conditions}, and thus cannot easily determine $\lambda$ or its sign here. We will consider all cases for $\lambda$ until we can establish the sign.} After integration by parts, this is equivalent to the ODE
\[
\phi^{(4)} - \lambda (\phi'' + 1) = 0 \qquad\text{on }\{\phi_R>0\}.
\]
Since $\phi_R$ is continuous, the set $\{\phi_R>0\}$ is open, and we can focus on an individual connected component \com{of $\{\phi_R>0\}$.} Clearly, the Lagrange multiplier can vary from connected component to connected component, \com{and the arguments showing that $\{\ol u >0\}$ is connected (assuming $\ol u$ exists) cannot be used to show that $\{\phi_R>0\}$ is connected.}

{\bf Solving the Euler-Lagrange equation.} \com{We fix a single connected component of $\{\phi_R>0\}$ for now which we translate to coincide with the interval $(-r,r)$ for some $r\leq R$.} We can distinguish two cases: $\lambda=0$ and $\lambda\neq 0$. \com{The solutions to these Euler-Lagrange equations on an interval $(-r,r)$ for a given constant $\lambda$ which may arise as profiles of minimizers $\phi_R$ will be denoted by $\phi$ for brevity. In the following, we will understand the different functions $\phi$ that may arise and find conditions on $\lambda, r$.} If $\lambda =0$, $\phi^{(4)}=0$ and thus any solution $\phi$ is a third degree polynomial. Knowing that $\phi(r) = \phi(-r) = \phi'(r) = \phi'(-r) = 0$ since $\phi \in W^{2,2}_0(-r,r)$, we find that $\phi$ can only be the constant zero function. Hence $\lambda\neq 0$. 

If $\lambda\neq 0$, the ODE seizes to be homogeneous. A particular solution of the equation is given by
\[
\bar \phi_R(x) = -\frac{x^2}2
\]
and the homogeneous ODE
\[
(\phi'' - \lambda\phi)'' = \phi^{(4)} - \lambda \phi'' = 0
\]
has the general solution
\[
\phi(x) = \begin{cases} a + bx + \alpha \cos(\mu x)  + \beta \sin(\mu x) &\lambda<0\\
	 a + bx + \alpha \cosh(\mu x)  + \beta \sinh(\mu x) &\lambda>0\end{cases}
\]
where $\mu= \sqrt{|\lambda|}$. Adding the particular and the general solution, we have a natural decomposition of $\phi$ into an odd and an even part
\[
\phi_{even}(x) = \begin{cases} a + \alpha \,\cos(\mu x) - \frac{x^2}2&\lambda<0\\ a + \alpha \,\cosh(\mu x) - \frac{x^2}2&\lambda>0\end{cases}
, \qquad \phi_{odd}(x) = \begin{cases} bx + \beta\,\sin(\mu x)&\lambda<0\\ bx + \beta\,\sinh(\mu x) &\lambda>0\end{cases}.
\]
We denote $\phi^{(k)}_{even}$ as the $k$-th derivative of the even part (not the even part of the $k$-th derivative).
Due to symmetry, both the even and the odd part of $\phi$ need to satisfy the matching conditions
\[
\phi_{even}(r) = \phi_{even}'(r) = \phi_{odd}(r) = \phi_{odd}'(r)\ \com{ = 0}
\]
separately, i.e.\ 
\[
\begin{array}{rcl}
0 &= \phi_{even}(r)
	&= \begin{cases} a + \alpha \,\cos(\mu r) - \frac{r^2}2 &\lambda<0\\
		 a + \alpha \,\cosh(\mu r) - \frac{r^2}2 &\lambda>0\end{cases}\\
0&= \phi_{even}'(r)
	&= \begin{cases} -\alpha\mu\,\sin (\mu r) - r &\lambda<0\\ 
		\alpha\mu\,\sinh (\mu r) - r &\lambda>0\end{cases}\\
0 &= \phi_{odd}(r)
	&= \begin{cases} br + \beta \sin(\mu r)&\lambda<0\\
		br + \beta \sinh(\mu r) &\lambda>0\end{cases}\\
0 &= \phi'_{odd}(r)
	&= \begin{cases} b + \beta\mu\,\cos(\mu r)&\lambda<0\\
		b + \beta\mu\,\cosh(\mu r) &\lambda>0\end{cases}
\end{array}
\]
For the even part, we note that
\[
\alpha = \begin{cases}\frac{- r}{\mu\,\sin(\mu r)} &\lambda<0\\ \frac{r}{\mu\,\sinh(\mu r)} &\lambda>0\end{cases}, \qquad a = \begin{pde}
\frac{r^2}2 - \alpha\cos(\mu r) &= \frac{r^2}2 + \frac{r}\mu\,\cot(\mu r)  &\lambda<0\\
\frac{r^2}2 -\alpha\cosh(\mu r) &= \frac{r^2}2  - \frac{r}\mu\coth(\mu r) &\lambda>0
\end{pde}
\]
Thus for the even part, the matching conditions can be satisfied whenever $\mu r\notin \pi \Z$ if $\lambda<0$ and always if $\lambda>0$.

{\bf Symmetry.} For the odd part, the matching conditions imply that 
\[
\frac{b}\beta = \frac{-\sin(\mu r)}r = -\mu\,\cos(\mu r)\qquad\Ra\qquad \frac{\sin(\mu r)}{\mu r} = \cos(\mu r)
\]
if $\beta\neq 0$, $\lambda<0$ and 
\[
\frac{b}\beta = \frac{-\sinh(\mu r)}r = -\mu\,\cosh(\mu r)\qquad\Ra\qquad \frac{\sinh(\mu r)}{\mu r} = \cosh(\mu r)
\]
if $\beta\neq 0$ and $\lambda>0$. If $\lambda >0$, the condition can never be satisfied because 
\[
\frac{\sinh(\rho)}\rho = \sum_{n=0}^\infty \frac{x^{2n}}{(2n+1)!}  < \sum_{n=0}^\infty \frac{x^{2n}}{(2n)!} = \cosh(\rho)
\]
for all $\rho>0$ while for $\lambda<0$, there are countably many solutions of the equation $\frac{\sin\rho}\rho - \cos\rho =0$. In the following, we will generally use the parameter $\rho = \mu r$ instead of $\mu = \frac{\rho}r$ as it simplifies many expressions. Note that, if $\frac{\sin\rho}\rho = \cos\rho$, then
\begin{align*}
\phi_{even}''(r) &= -1 - \alpha\,\mu^2\cos(\mu r)\\
	&= -1 + \frac{r}{\mu \sin(\mu r)}\,\mu^2\cos(\mu r)\\
	&= -1 + \frac{\mu r}{\sin(\mu r)}\,\cos(\mu r)\\
	&= 0\\
\phi_{odd}''(r) &= -\beta\,\mu^2\,\com{\sin}(\mu r)\\
	&\neq 0
\end{align*}
unless $\beta =0$. Thus, if $\phi''(r) = \phi_{odd}''(r) >0$, then $\phi''(-r) <0$ and vice versa. Since $\phi(\pm r) = \phi'(\pm r) =0$, this would imply that $\phi$ cannot be non-negative. Thus, whether $\lambda>0$ or $\lambda<0$, we have shown that $\phi$ has to be an even function, $\phi = \phi_{even}$.

Note that the odd part is excluded for two different reasons, depending on whether $\lambda>0$ or $\lambda<0$. If $\lambda>0$, the operator $\phi^{(4)} - \lambda\phi''$ is positive definite on $W^{2,2}$ (since we integrate by parts twice in the first and once in the second term), so that it cannot have a non-trivial kernel and the purely even solution of the boundary value problem that we constructed is unique. If $\lambda<0$, it depends on the relationship between $\lambda$ and the length of the interval whether we can find a non-trivial solution to the homogeneous problem. It is only the sign constraint on $\phi$ that excludes them.

{\bf Calculating energy and length.}
We calculate energy and length of $\phi$, first in the case $\lambda<0$:
\begin{align*}
\E(\phi)&= \int_{-r}^r |\phi''|^2\dx\\
	&= \int_{-r}^r \left|\alpha\,\mu^2\cos(\mu x) + 1\right|^2\dx\\
	&= \int_{-r}^r 1 + 2\alpha\mu^2\cos(\mu x) + \alpha^2\mu^4\,\cos^2\mu x\dx\\
	&= 2\left[ r + 2\alpha\mu\sin(\mu r) + \alpha^2\mu^4 \left(\frac r2 + \frac{\sin (\mu r)\,\cos(\mu r)}{2\mu}\right)\right]\\
	&= \left[ 2 + \frac{r^2}{\mu^2\,\sin^2(\mu r)}\,\mu^4\right] r + 4(-r) + \left(\frac{-r}{\mu\,\sin(\mu r)}\right)^2 \mu^3 \sin(\mu r)\,\cos(\mu r)\\
	&= \left(\frac{(\mu r)^2}{\sin^2(\mu r)} + \frac{\mu r}{\sin \mu r} \cos(\mu r) -2 \right)\,r\\
	&= \left(\frac{\rho^2}{\sin^2(\rho)} + \frac{\rho}{\sin\rho}\,\cos(\rho)-2\right)r
\end{align*}
where we substituted the variable $\rho$ for $\mu r$. Similarly, we compute
\begin{align*}
L(\phi) &= \int_{-r}^r \frac{(\phi')^2}2 - \phi\dx\\
	&=  \int_{-r}^r \frac12 \left[ -x - \alpha\mu\,\sin(\mu x)\right]^2 - \left[ a + \alpha \,\cos\mu x - \frac{x^2}2\right]\dx\\
	&= \int_{-r}^r \frac{x^2}2 + \alpha\mu x\,\sin(\mu x) + \frac12\,\alpha^2\mu^2\,\sin^2(\mu x) - a - \alpha\,\cos\mu x + \frac{x^2}2\dx\\
	&= 2\left[ \frac{r^3}{3} + \alpha \,\frac{\sin(\mu r) - \mu r \,\cos(\mu r)}{\mu} + \frac{\alpha^2\mu^2}2\left(\frac r 2 - \frac{\sin(\mu r)\,\cos(\mu r)}{2\mu}\right) - a r - \frac{\alpha}\mu\,\sin(\mu r) \right]\\
	&=  2\bigg[ \frac{r^3}{3} + \frac{r^2\,\cos(\mu r)}{\mu\,\sin(\mu r)} +  \left(\frac{r}{\mu\,\sin(\mu r)}\right)^2\frac{\mu^2}4r  - \left(\frac{r}{\mu\,\sin(\mu r)}\right)^2\frac{\mu}4\,\sin(\mu r) \cos(\mu r)\\
		&\qquad- \left(\frac{r^2}2 + \frac{r}\mu\,\cot(\mu r)\right) r \bigg]\\
	&= 2\left[\frac{r^3}3 + \frac{r^2}\mu\cot (\mu r) + \frac{r^3}{4\,\sin^2(\mu r)} - \frac{r^2}{4\mu}\cot (\mu r) - \frac{r^3}2 - \frac {r^2}\mu\cot(\mu r)\right]\\
	&= \left[ - \frac r3 + \frac{r}{2\,\sin^2\rho} - \frac{\cot \rho}{2\,\frac\rho r}\right]r^2\\
	&= \left[-\frac13 + \frac1{2\sin\rho}\left(\frac1{\sin \rho}- \frac{\cos\rho}\rho\right)\right]r^3
\end{align*}
Finally, we need to repeat the calculations for the case of a positive Lagrange multiplier. A direct calculation yields
\begin{align*}
\E(\phi) 
	&= \left[\frac{\rho^2}{\sinh^2(\rho)} + \frac{\rho}{\sinh(\rho)}\cosh(\rho) -2\right]r\\
L(\phi) 
	&= \left[ - \frac13 + \frac{1}{2\sinh(\rho)}\left( \frac{\cosh(\rho)}\rho - \frac1{\sinh(\rho)}\right)\right]\,r^3
\end{align*}
with calculations very similar to the case $\lambda<0$ -- the only differences are signs that need to be carefully taken into account.

{\bf Estimating $r$.} Assume for a contradiction that over intervals $[-R_k,R_k]$, we have a sequence of minimisers \com{$\phi_{k}$ such that $\{\phi_k>0\}$ has at least one connected component with diverging length. As before, we translate these components to coincide with an interval $(-r_k, r_k)$ with $r_k\to\infty$.} Then necessarily
\[
\lim_{k\to \infty} \left(\frac{\rho_k^2}{\sin^2 \rho_k} + \frac{\rho_k}{\sin\rho_k}\,\cos(\rho_k)-2\right) =0
\]
or -- if $\lambda>0$ for infinitely many minimisers -- 
\[
\lim_{k\to \infty} \left(\frac{\rho_k^2}{\sinh^2 \rho_k} + \frac{\rho_k}{\sinh\rho_k}\,\cosh(\rho_k)-2\right) =0.
\]
Consider the second case first. Then we know that $\cosh(\rho)>\sinh(\rho)$ for all $\rho>0$
\begin{align*}
\frac{\rho_k^2}{\sinh^2\rho_k} + \frac{\rho_k}{\sinh\rho_k}\,\cosh(\rho_k)-2 &\geq \rho_k -2 
\end{align*}
which means that $0<\rho_k<3$ for almost all $k\in\N$. By compactness, there exists $\ol \rho\in[0,3]$ such that $\rho_k\to\ol\rho$ (up to a subsequence) and
\[
\frac{\ol\rho}{\sinh^2\ol\rho} + \frac{\ol\rho}{\sinh\ol\rho}\,\cosh\ol\rho-2=0.
\]
If $\ol\rho>0$, then the energy of a function $\phi$ associated to $\ol\rho$ and any $r>0$ would be zero, but as this is not the case, $\ol\rho$ must be $0$. 

In the first case, on the other hand, we know that $\left|\frac\rho {\sin\rho}\right|\geq \rho$ for all $\rho>0$ and since 
\[
X^2 -X -2 \geq 1
\]
for all $|X|\geq 2$, we find that $\rho \leq \left|\frac\rho{\sin\rho}\right| \leq 2$. We reach the same conclusion as before. Now observe that if $\lambda<0$ we have
\begin{align*}
\phi(0) &= a+\alpha\\
	&= \frac{r^2}2 + \frac r\mu\,\cot(\mu r) - \frac{r}{\mu\,\sin(\mu r)}\\
	&= \left(\frac12 + \frac{\cos \rho -1}{\rho\,\sin\rho} \right)r^2\\
	&= \left(\frac12 + \frac{-\frac{\rho^2}2 + \frac1{24}\rho^4 + O(\rho^6)} {\rho\,(\rho - \frac{\rho^3}6 + O(\rho^5))} \right)r^2\\
	&= \left(-\frac{\rho^2}{24} + O(\rho^4)\right)\,r^2,
\end{align*}
so $\phi(0)<0$ if $\rho$ is too small which poses a contradiction. Similarly, we can compute that
\[
L(\phi) = \left(-\frac{2}{45}\,\rho^2 + O(\rho^4)\right)r^3
\]
for small $\rho$ if $\lambda>0$, so if $\lambda>0$ and $\rho$ is very small, we find that $L(\phi)<0$, leading to a contradiction again. We conclude that $r$ is uniformly bounded (and $\rho$ is uniformly bounded away from $0$).

{\bf Minimsers on the real line.} Let $\eps>0$ and $\psi \in M$ be a function such that $\E(\psi) < \Theta + \eps$. Since $\psi\in M$ is compactly supported, we see that there exists $R>0$ such that $\supp(\psi) \subset (-R,R)$ and thus in particular 
\[
\inf_{\phi \in \ol M_R}\E(\phi) \leq \E(\psi) <\Theta +\eps.
\]
We conclude that letting $R\to \infty$, we recover the original energy infimum:
\[
\lim_{R\to\infty} \inf_{\phi \in \ol M_R}\E(\phi) = \Theta
\]
where the limit exists since the quantity is monotone decreasing in $R$. Now let $\phi_k$ be the minimiser of $\E$ in $W^{2,2}_0(-k,k)$. We know that $\lim_{k\to \infty}\E(\phi_k) =\Theta$, so by Lemma \ref{lemma mass confinement}, there exists \com{one} connected component $I_k$ of $\{\phi_k>0\}$ such that
\[
\int_{I_k} \frac{(\phi_k')^2}2 - \phi_k\dx \to 1.
\] 
After a translation, we have $I_k= (-r_k, r_k)$ and we introduce the restriction $\psi_k = \phi_k|_{I_k}$. We observe that $\psi_k$ is a function of the type $\psi_k = \phi_{\rho_k, r_k}^\pm$ for suitable parameters $\rho_k, r_k$ and a choice $\pm$ of either $\lambda>0$ or $\lambda<0$. Since $r_k\not\to+\infty$ by the previous step in the proof, we find that $\psi_k\in W^{2,2}_0(-R,R)$ for some suitably large $R$ and
\[
\int_{-R}^R|\psi_k''|^2 \dx \leq \int_\R|\phi_k''|^2\dx \leq \Theta + \eps_k
\]
where $\eps_k\to 0$. On the bounded set $(-R,R)$ and with zero boundary values, this controls the entire $W^{2,2}$-norm. In particular, there exists a function $\ol u\in W^{2,2}_0(-R,R)$ such that $\psi_k\wto \ol u$ in $W^{2,2}(-R,R)$ and thus
\[
\E(\ol u) = \int_{-R}^R |\ol u''|^2\dx \leq \liminf_{k\to \infty}\int_{-R}^R|\psi_k''|^2 \dx \leq \Theta.
\]
Furthermore, since $\psi_k\to \ol u$ strongly in $W^{1,2}(-R,R)$, we find that
\[
L(\ol u) = \lim_{k\to \infty} L(\psi_k) = \lim_{k\to \infty} \int_{I_k} \frac{(\phi_k')^2}2 - \phi_k \dx = 1
\]
by our choice of $I_k$. In total, $\ol u$ satisfies 
\[
\ol u \in W^{2,2}_0(-R,R) \subset W^{2,2}(\R)\cap L^1(\R), \quad L(\ol u) = 1\qquad\Ra \quad \ol u\in \ol M
\]
and $\E(\ol u) = \Theta$ which means that $\ol u$ is \com{indeed} a minimiser of $\E$ in $\ol M$ \com{as the naming in accordance with Lemma \ref{lemma minimiser conditions} suggests}.

{\bf Direct consequences.} Since the energy $\E$ admits a minimiser $\ol u\in \ol M$, we conclude from the arguments above that $\ol u$ is compactly supported and thus can be found as the minimiser of $\E$ over $\ol M_R$, that the Lagrange multiplier is $\lambda = - \frac{\Theta}6 <0$ and that the set $\{\ol u>0\}$ is connected.

{\bf Smoothness.} We quickly observe that the Lagrange multiplier $\lambda\neq 0$ satisfies
\[
0\neq \lambda = \frac1{2r}\int_{-r}^r \lambda\dx = \frac1{2r}\int_{-r}^r \phi^{(4)} - \lambda\phi''\dx = \frac{\phi^{(3)}(r) - \phi^{(3)}(-r)}{2r}
\]
since $\phi'(\pm r) = 0$ as $\phi$ is globally $C^1$-smooth. This means that either $\lim_{x\to r^-}\phi^{(3)}(x)\neq 0$ or $\lim_{x\to -r^+}\phi^{(3)}(x) \neq 0$ which implies immediately that $\phi$ cannot be globally $C^3$-smooth.

Note that $\phi(x) = 0$ for all $x>r$, so if $\phi$ is $C^2$-smooth, then necessarily $\lim_{x\to r^+}\phi''(x) = 0$. We calculate
\[
\lim_{x\to r^-}\phi''(x) = -1 - \alpha\mu^2\,\cos(\mu r) = -1 + \mu^2\,\frac{r}\mu\,\cot(\mu r) = 0 \qquad \LRa \qquad \frac{\sin(\mu r)}{\mu r} = \cos(\mu r)
\]
as we already encountered when considering the odd part of $\phi$.
Let us now use the characterisation of the Lagrange multiplier from Lemma \ref{lemma minimiser conditions} to compute that
\begin{align*}
\mu^2 &= \left(\frac \rho r\right)^2\\
	&= \frac{\E(\phi)}{6\,L(\phi)}\\
	&= \frac{\left(\frac{\rho^2}{\sin^2(\rho)} + \frac{\rho}{\sin\rho}\,\cos(\rho)-2\right)r}{6\left[-\frac13 + \frac1{2\sin\rho}\left(\frac1{\sin \rho}- \frac{\cos\rho}\rho\right)\right]\,r^3}\\
\end{align*}
which means that
\[
\rho^2 = \frac{\frac{\rho^2}{\sin^2(\rho)} + \frac{\rho}{\sin\rho}\,\cos(\rho)-2}{6\left[-\frac13 + \frac1{2\sin\rho}\left(\frac1{\sin \rho}- \frac{\cos\rho}\rho\right)\right]}
\]
or 
\[
6\left[-\frac{\rho^2}3 + \frac{\rho}{2\sin\rho}\left(\frac\rho {\sin \rho}- \cos \rho\right)\right] = \frac{\rho^2}{\sin^2(\rho)} + \frac{\rho}{\sin\rho}\,\cos(\rho)-2.
\]
Further algebra shows that this is equivalent to
\[
- 2\rho^2 + 3 \left(\frac\rho {\sin\rho}\right)^2 - 3\cos\rho\left(\frac \rho {\sin\rho}\right) = \left(\frac \rho{\sin\rho}\right)^2 + \cos\rho \left(\frac \rho{\sin\rho}\right)-2
\]
and finally
\[
2\left[ \left(\frac\rho {\sin\rho}\right)^2 - 2\,\cos \rho\left(\frac\rho {\sin\rho}\right) +1-\rho^2 \right] = 0.
\]
We compute that
\begin{align*}
&&\rho^2 +(1-\rho^2)\sin^2\rho &= 2\rho\cos\rho\sin\rho\\
\Ra&&\rho^4 + 2\rho^2(1-\rho^2)\sin^2\rho + (1-\rho^2)^2\sin^4\rho &= 4\rho^2 (1-\sin^2\rho)\sin^2\rho\\
\LRa&&\rho^4 + 2\rho^2(1-\rho^2)\sin^2\rho -4\rho^2\sin^2\rho + (1-2\rho^2+\rho^4)\sin^4\rho +4\rho^2\sin^4\rho &=0\\
\LRa&&\rho^4 + 2\com{\rho^2}(-1-\rho^2)\sin^2\rho + (1+\rho^2)^2\sin^4\rho &=0\\
\LRa&&\big[\rho^2 - (1+\rho^2)\sin^2\rho\big]^2 &= 0
\end{align*}
so neccessarilly 
\[
\sin^2\rho = \frac{\rho^2}{1+\rho^2}\qquad\Ra \qquad \cos^2\rho = 1- \sin^2\rho = \frac{1}{1+\rho^2} = \frac{\sin^2\rho}{\rho^2}
\]
such that $\cos\rho = \frac{\sin\rho}\rho$ is satisfied at least up to a sign. If $\cos \rho = \frac{\sin\rho}\rho$, we get
\[
\left(\frac\rho {\sin\rho}\right)^2 - 2\,\cos \rho\left(\frac\rho {\sin\rho}\right) +1-\rho^2 = 
\left(\frac\rho{\sin\rho}\right)^2 - 1-\rho^2 = 
\frac{\rho^2}{\frac{\rho^2}{1+\rho^2}} - (1+\rho^2) = 0,
\]
so the original equation is satisfied. If, on the other hand, $\cos\rho = -\frac{\sin\rho}\rho$, we find a contradiction assuming that
\[
0= \left(\frac\rho {\sin\rho}\right)^2 - 2\,\cos \rho\left(\frac\rho {\sin\rho}\right) +1-\rho^2= \left(\frac\rho {\sin\rho}\right)^2 + 3-\rho^2\qquad \Ra\qquad \sin^2\rho = \frac{\rho^2}{\rho^2-3}\neq \frac{\rho^2}{\rho^2+1}.
\]
So the minimiser must satisfy $\tan\rho = \rho$ which by our previous computations implies that $\ol u$ is $C^2$-smooth. Since $\ol u$ is $C^\infty$-smooth on $\{\ol u>0\}$ and $\{\ol u=0\}^\circ$, we find that $\ol u$ has a bounded weak third derivative, i.e. $\ol u \in W^{3,\infty}(\R) = C^{2,1}(\R)$.

{\bf Finding the minimiser.} The minimiser $\ol u$ of $\E$ in $\ol M$ is given by $\ol u = \phi_{\rho,r}$ for parameters $\rho,r$ which satisfy $\tan\rho = \rho$ and a fortiori $\sin^2\rho = \frac{\rho^2}{1+\rho^2}$. 
We can therefore re-write length and energy as
\begin{align*}
L(\phi_{\rho,r}) &= \left[-\frac13 + \frac1{2}\left(\frac1{\sin^2 \rho}- \frac{\cos\rho}{\rho\sin\rho}\right)\right]r^3\\
	&= \left[ -\frac13 +\frac12\left(\frac{1}{\frac{\rho^2}{1+\rho^2}} - \frac1{\rho^2}\right) \right]r^3\\
	&= \left[-\frac13 + \frac{1+\rho^2}{2\rho^2}-\frac1{2\rho^2}\right]r^3\\
	&= \left[-\frac1{\com3} + \frac12 + \frac1{2\rho^2} - \frac1{2\rho^2}\right]r^3\\
	&= \frac{r^3}6\\
\E(\phi_{\rho,r}) &= \left[\frac{\rho^2}{\sin^2(\rho)} + \frac{\rho}{\sin\rho}\,\cos(\rho)-2\right]r\\
	&= \left[\frac{\rho^2}{\frac{\rho^2}{1+\rho^2}} + 1-2\right]r\\
	&= \left[1+\rho^2-1\right]r\\
	&= \rho^2r.
\end{align*}
Given $\rho$, we need to find $r$ such that 
\[
1 = L(\phi_{\rho,r}) = \frac{r^3}6 =1 \qquad \Ra \quad r = 6^{1/3}
\]
and calculate the energy 
\begin{align*}
\E(\phi_{\rho,r}) &=\rho^2\,r = 6^{1/3}\rho^2
\end{align*}
As this function is increasing in $\rho$, we need to find the {\em first positive solution} $\ol \rho$ of $\tan(\ol \rho) = \com{\ol\rho}$ for the global minimiser. Since $\tan(\rho)>\rho$ for $\rho \in (0,\pi/2)$ and $\tan(\rho)<0$ for $\rho\in(\pi/2, \pi)$, we find that $\pi < \ol \rho < \frac{3\pi}2$. Numerically, we find 
\[
\ol \rho \approx 4.4934.
\]
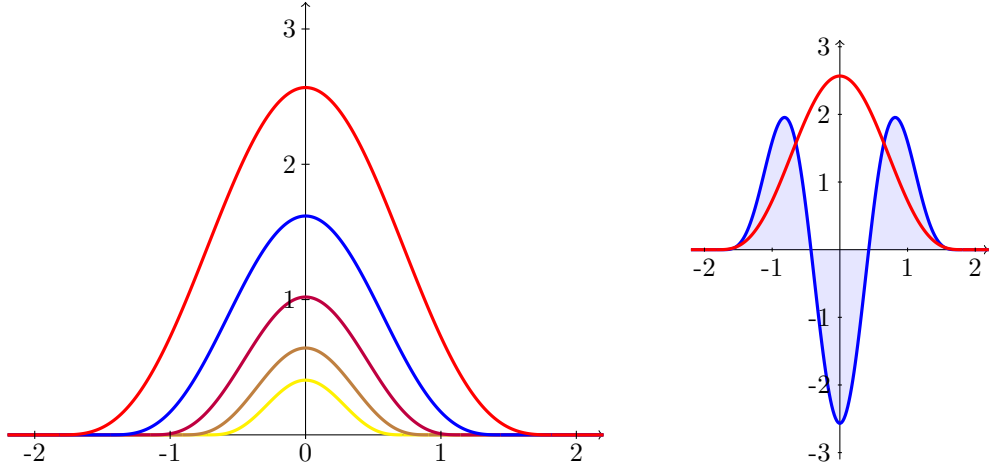
\begin{figure}
\begin{center}
\begin{tikzpicture}[scale=1.8]
\node at (-2.501, -0.1){};
\node at (2.501, 3.201){};

\draw[->] (-2.2,0) -- (2.2,0);
\draw[->] (0,-0.03) -- (0,3.2);

\draw[very thick, yellow, domain = 0.7211 : 2.2] plot ({\x}, {0});
\draw[very thick, yellow, domain = 0.7211 : 2.2] plot ({-\x}, {0});

\draw[very thick, brown, domain = 0.9086 : 2.2] plot ({\x}, {0});
\draw[very thick, brown, domain = 0.9086 : 2.2] plot ({-\x}, {0});

\draw[very thick, purple, domain = 1.1447 : 2.2] plot ({\x}, {0});
\draw[very thick, purple, domain = 1.1447 : 2.2] plot ({-\x}, {0});

\draw[very thick, blue, domain = 1.4422 : 2.2] plot ({\x}, {0});
\draw[very thick, blue, domain = 1.4422 : 2.2] plot ({-\x}, {0});

\draw[very thick, red, domain = -2.2 : -1.8171] plot ({\x}, {0});
\draw[very thick, red, domain = 1.8171 : 2.2] plot ({\x}, {0});
\draw[very thick, red, domain = -1.8171 : 1.8171, samples = 200] plot ({\x}, {1.8145 - 0.5*\x*\x + 0.7528 * cos (180/pi*2.4728 *\x)});

\draw[very thick, blue, domain = -1.4422 : 1.4422, samples = 200] plot ({\x}, {2^(-2/3)* (1.8145 - 0.5*2^(2/3)*\x*\x + 0.7528 * cos (2^(1/3)* 180/pi*2.4728 *\x))});

\draw[very thick, purple, domain = -1.1447 : 1.1447, samples = 200] plot ({\x}, {4^(-2/3)* (1.8145 - 0.5*4^(2/3)*\x*\x + 0.7528 * cos (4^(1/3)* 180/pi*2.4728 *\x))});

\draw[very thick, brown, domain = -0.9086 : 0.9086, samples = 200] plot ({\x}, {8^(-2/3)* (1.8145 - 0.5*8^(2/3)*\x*\x + 0.7528 * cos (8^(1/3)* 180/pi*2.4728 *\x))});

\draw[very thick, yellow, domain = -0.7211 : 0.7211, samples = 200] plot ({\x}, {16^(-2/3)* (1.8145 - 0.5*16^(2/3)*\x*\x + 0.7528 * cos (16^(1/3)* 180/pi*2.4728 *\x))});

\foreach \x in {1,..., 3} 
	{
	\draw[]( -0.03,\x) -- ++ (0.06,0);
	\node [left] at (0,\x) {\x};
	}
\foreach \x in {-2 , ..., 2} 
	{
	\draw[](\x, -0.03) -- ++ (0, 0.06);
	\node [below] at (\x,0) {\x};
	}

\end{tikzpicture}
\begin{tikzpicture}[scale=0.9]
\node at (-2.501, -0.1){};
\node at (2.501, 3.201){};

\draw[->] (-2.2,0) -- (2.2,0);
\draw[->] (0,-3.1) -- (0,3.1);

\draw[very thick, blue, domain = -2.2 : -1.8171] plot ({\x}, {0});
\draw[very thick, blue, domain = 1.8171 : 2.2] plot ({\x}, {0});
\draw[very thick, blue, domain = -1.8171 : 1.8171, samples = 200] plot ({\x}, {0.5* (\x +2.4728*0.7528*sin(180/pi * 2.4728*\x) )^2 - 1.8145 + 0.5*\x*\x  - 0.7528*cos (180/pi * 2.4728 *\x)  });
\fill[opacity=0.1, blue, domain = -1.8171 : 1.8171, samples = 200] plot ({\x}, {0.5* (\x +2.4728*0.7528*sin(180/pi * 2.4728*\x) )^2 - 1.8145 + 0.5*\x*\x  - 0.7528*cos (180/pi * 2.4728 *\x)  });


\draw[very thick, red, domain = -2.2 : -1.8171] plot ({\x}, {0});
\draw[very thick, red, domain = 1.8171 : 2.2] plot ({\x}, {0});
\draw[very thick, red, domain = -1.8171 : 1.8171, samples = 200] plot ({\x}, {1.8145 - 0.5*\x*\x + 0.7528 * cos (180/pi*2.4728 *\x)});

\foreach \x in {1,..., 3} 
	{
	\draw[]( -0.03,\x) -- ++ (0.06,0);
	\node [left] at (0,\x) {\x};
	}
\foreach \x in {-3,..., -1} 
	{
	\draw[]( -0.03,\x) -- ++ (0.06,0);
	\node [left] at (0,\x) {\x};
	}	
	
\foreach \x in {-2 , -1} 
	{
	\draw[](\x, -0.03) -- ++ (0, 0.06);
	\node [below] at (\x,0) {\x};
	}
\foreach \x in {1 , 2} 
	{
	\draw[](\x, -0.03) -- ++ (0, 0.06);
	\node [below] at (\x,0) {\x};
	}

\end{tikzpicture}
\end{center}
\caption{\label{figure minimiser} Left: the minimiser $\ol u$ of $\E$ in $\ol M$ and its rescalings $\ol u_\delta$ for $\delta = 2^{-1}$ (blue), $\delta=2^{-2}$ (purple), $\delta = 2^{-3}$ (brown) and $\delta = 2^{-4}$ (yellow). Right: $\ol u$ (red) and the length integrand $\frac{(\ol u')^2}2 - \ol u$ (blue).}
\end{figure}

We then calculate
\[
\begin{array}{rll}
\Theta &= 6^{1/3}\rho^2 
	&\approx 36.6890\\ 
\ol r &= 6^{1/3}
	&\approx 1.8171\\ 
\ol \mu &= \frac{\ol \rho}{\ol r}
	&\approx 2.4728\\ 
\ol \alpha &= - \frac{\ol r^2}{\ol \rho\,\sin(\ol \rho)} 
	&\approx 0.7528\\ 
\ol a &= \left(\frac12 + \frac{\cot \ol \rho}{\ol \rho}\right) \ol r^2
	&\approx 1.8145. 
\end{array}
\]
It is easy to see that $\ol u'(-r) = \ol u'(0)$ and if $\ol u'$ has a local extremum in $(-r,0)$, then
\[
\ol u''(x) = 0 \qquad \Ra\quad -\alpha\mu^2\cos(\mu x) -1 = 0\qquad\Ra\quad \cos(\mu x) = \frac{-1}{\alpha\mu^2} <0.
\]
Since $0 < |\mu x|< \mu r < \frac{3\pi}2$, this \com{last} equation has at most two solutions in $[- r,0]$. As we already know, one of the solutions is $-r$ itself since $\ol u''(-r) = 0$ and observing that $\ol u^{(3)}(-r)>0$, we get that $\ol u$ is increasing at $-r$. This means that $\ol u'$ has a local maximum and no local minima in $(-r,0)$ and thus that $\ol u$ is increasing on $(-r,0)$. The same argument shows that $\ol u$ is decreasing on $(0,r)$. In particular, we deduce that $\ol u>0$ on $(-r,r)$.
\end{proof}

\subsection{Functions of Low Energy}\label{section low energy}

\com{In this section, we close the gap in the proof above by establishing that $\Theta>0$. This} is sufficient to establish the order of energy scaling in Theorem \ref{theorem short length}, but not to find the leading order coefficient explicitly.

\begin{lemma}\label{lemma basic variational}
Let $0\leq \phi \in W^{2,2}(\R)\cap L^1(\R)$ and define 
\[
I_+ = \left\{\frac{(\phi')^2}2 - \phi > 0\right\} = \left\{x\in \R\:\bigg|\: \frac{(\phi')^2(x)}2 - \phi(x) > 0\right\}.
\]
Then the following hold.
\begin{enumerate}
\item $\H^1(I^+) \leq \int_{\R}|\phi''|^2\ds$.

\item The height and slope of $\phi$ are related by
\[
|\phi'(x)|^3\leq \left( 3\int_\R|\phi''|^2\dy\right) \phi(x)\qquad \forall\ x\in\R.
\]

\item The length integrand is bounded by
\[
\frac{(\phi')^2}2 - \phi \leq \frac{\left(\int_\R|\phi''|^2\ds\right)^2}6.
\]
\end{enumerate}
\end{lemma}

\begin{proof}
{\bf First Property.} Let $\phi \in \ol M$. Since $\phi, \phi'$ are continuous, the set  
\[
I_+ = \big\{(\phi')^2 - 2\phi > 0\big\} = \left\{x\in \R\:\bigg|\: \frac{(\phi')^2(x)}2 - \phi(x) \geq 0\right\}  = \bigcup_{n\in \Z} (a_n,b_n)
\]
is a union of disjoint open intervals with endpoints $a_n, b_n$ satisfying
\[
\frac{(\phi')^2}2(a_n) - \phi(a_n) = \frac{(\phi')^2}2(b_n) - \phi(b_n) = 0.
\]
Since $\R$ is second countable, the union is at most countable.
Clearly, inside an interval there cannot be any point where $\phi'$ vanishes, so we find that on any interval $(a_n,b_n)$ the function $\phi$ is either increasing or decreasing. In particular, $\phi(a)>0$ or $\phi(b)>0$. Let us fix $n$ and consider one such interval of length $\ell = b-a$ where we assume without loss of generality that $\phi$ is increasing and that $\phi(\com{b})>0$. We denote $h:= \phi(a)$ and $\rho := \sqrt{2 h}$ and use a translation to normalize $a = \rho$, $b= \rho+\ell$. Since the ODE 
\[
\begin{pde}
f' &= \sqrt{2f} &x>\rho \\
f &= h \: \left(= \frac{\rho^2}2\right) &x=\rho
\end{pde}
\] 
is solved by $f(x) = \frac{x^2}2$ and $\phi' \geq \sqrt{2\phi}$ on the interval $(a,b)$, we find that
\[
\phi(x) \geq \frac{x^2}2 \qquad\forall\ x\in [a,b]
\]
by the comparison principle for ODEs. 
It follows that 
\[
\phi'(a) = \rho, \qquad \phi'(b) = \sqrt{2\,\phi(b)} \geq \sqrt{2\,\frac{(\rho+\ell)^2}2} = \rho+\ell
\]
whence 
\[
\ell = \phi'(b) - \phi'(a) = \int_a^b \phi'' \ds \leq |b-a|^{\frac12} \left(\int_a^b |\phi''|^2 \ds\right)^\frac12\quad\Ra\quad \ell \leq \int_a^b |\phi''|^2\ds.
\]
Adding up the terms over the intervals $(a_n, b_n)$, we find that
\[
\H^1(I_+) \leq \int_{I^+}|\phi''|^2 \ds \leq \int_{\R}|\phi''|^2\ds.
\]

{\bf Second property.} 
Denote $\Xi:= \E(\phi)^{1/2}$. Without loss of generality, we may assume that $\phi(0)>0$. We denote $\phi(0) = h$ and $\phi'(0) =\com{m}$ Denoting $l_{h,\com{m}}(x) = h+ \com{mx}$ we observe that
\begin{align}
\big|\phi'(x) - \com{m}\big| &= \left|\int_0^x\phi''(s)\ds\right| &
	 &\leq \left(\int_0^x|\phi''(\com{s})|^2\com{\ds}\right)^\frac12 |x|^{1/2} &&\leq \Xi\,|x|^{1/2} \label{eq holder}\\
\big|\phi(x) - l_{h,\com{m}}(x)\big| &\leq \int_0^x \big|\phi'(s)-\com{m}\big|\ds \nonumber&
	& \leq \Xi\int_0^x |s|^{1/2}\ds &&= \frac{2\Xi}3\,|x|^{3/2}.\nonumber
\end{align}
Since $\phi\geq 0$, this gives us a compatibility condition on the height $h$ and the slope $\phi$:
\begin{align*}
0 &\leq \phi(x)\\
	&\leq l_{h,\com{m}}(x) + |\phi(x)-l_{h,\com{m}}(x)|\\
	&\leq h+ \com{m}x + \frac{2\Xi}3\,|x|^{3/2}.
\end{align*}
To find the minimum of the expression on the right, we may assume that $\com{m}<0$ so that the minimum is positive. Taking derivatives, we find that the minimum is assumed when
\[
\com{m} + \Xi\,|x|^{1/2} = 0\qquad \Ra\qquad |x| = \frac{\com{m}^2}{\Xi^2}\qquad \Ra\qquad h+ \com{m}x + \frac{2\,\Xi}3\,|x|^{3/2} = h + \frac{\com{m}^3}{\Xi^2} - \frac{2\Xi}3 \,\frac{\com{m}^3}{\Xi^3} = h + \frac{\com{m}^3}{3\,\Xi^2}.
\]
It follows that we have to require $h\geq \frac{|\com{m}|^3}{3\Xi^2}$ for $\phi$ to be positive, or in other words 
\[
\frac{|\phi'(0)|^3}{3\,\Xi^2} \leq \phi(0).
\]
Since we can repeat the argument at any point $x$, we have established the \com{second} property.

{\bf Step 3.} If $x\in I_+$ we have
\[
0 \leq \frac{(\phi')^2(x)}2 - \phi(x) \leq \left( \frac12  - \frac{|\phi'(x)|}{3\Xi^2}\right)\,(\phi')^2(x) \qquad \Ra \qquad |\phi'(x)| \leq \frac{3\,\Xi^2}2.
\]
If we want to maximise the length integrand, we may consider
\[
g(a) := \left(\frac12 - \frac{|a|}{3\Xi^2}\right)\,a^2
\]
and observe that $g$ has to have a global maximum at a point $a\neq 0$ and without loss of generality we may assume that $a>0$. Thus
\[
g'(a) = -\frac1{3\Xi^2}\,a^2 + \left(\frac12 - \frac{a}{3\Xi^2}\right)\,2a = - \frac{a^2}{\Xi^2} + a = \left(1 - \frac{a}{\Xi^2}\right)a
\]
which means that $a = \Xi^2$ and
\[
g(a) = \left(\frac12 - \frac{\Xi^2}{3\Xi^2}\right) \Xi^4 = \frac{\Xi^4}6.
\]
\end{proof}

\begin{corollary}\label{corollary theta>0}
From Lemma \ref{lemma basic variational}, it follows that $1 < 6^\frac 13 \leq \inf_{\phi\in M}\E(\phi)$.
\end{corollary}

\begin{proof}
Observe that
\begin{align*}
1 &= \int_\R \frac{(\phi')^2}2 - \phi\dx\\
	&\leq \int_{I_+} \frac{(\phi')^2}2 - \phi\dx\\
	&\leq \H^1(I_+)\cdot \sup_{x\in I_+} \left(\frac{(\phi')^2}2 - \phi\right)(x)\\
	&\leq \Xi^2 \cdot \frac{\Xi^4}6
\end{align*}
such that $\Xi^6\geq 6$.
\end{proof}

\begin{remark}
Note that the functional $L$ drives the fact that the $\|\phi\|_{L^1}$ and $\|\phi'\|_{L^2}^2$ are comparable. \com{We can use similar arguments as before to obtain the following non-linear norm-estimates for $\phi\in M$ with universal constants $C>0$.}
\[
\|\phi\|_{L^1} \leq C\,\|\phi''\|_{L^2}^6, \qquad \|\phi\|_{L^2} \leq C\,\|\phi''\|_{L^2}^{5}, \qquad \|\phi\|_{L^\infty}\leq C\,\|\phi''\|_{L^2}^4
\]
and
\[
\|\phi'\|_{L^2} \leq C\,\|\phi''\|_{L^2}^3\,, \qquad \|\phi'\|_{L^\infty} \leq C\,\|\phi''\|_{L^2}^2.
\]
\com{
We only briefly outline the proofs and leave the details to the reader.
\begin{enumerate}
\item At a maximum of $\phi$, we have $\phi'=0$ su that $\frac{(\phi')^2}2 -\phi<0$. If the maximum is very high, the domain where the length integrand is negative is very large and we generate a large amount of negative length. Since the total length is positive, this forces high energy similarly to Corollary \ref{corollary theta>0}. We thus establish the $L^\infty$-bound on $\phi$. 
\item The second point of Lemma \ref{lemma basic variational} allows us to obtain an $L^\infty$-bound on $\phi'$ from the $L^\infty$-bound on $\phi$.
\item The $L^1$-bound on $\phi$ can be obtained by considering the set $\tilde I^+ = \left\{x\in \R\:|\: \left(\phi'\right)^2(x) > \phi(x)\right\}$ and using a bound on the measure of $\tilde I^+$ like the bound on $I^+$ from \ref{lemma basic variational}. On the complement of $\tilde I^+$, we see that 
\[
\frac{(\phi')^2}2 - \phi \leq -\frac{\phi}2\leq 0
\]
is integrable.
\item The $L^2$-bound on $\phi$ follows by interpolation between $L^1$ and $L^\infty$ while the $L^2$-bound on $\phi'$ follows by integration by parts and using H\"older's inequality on $\phi\phi''$.
\end{enumerate}
}
\end{remark}

\begin{remark}
We have seen that sets of the form $\ol M^\Xi = \{\phi \in \ol M\:|\: \E(\phi)\leq \Xi^2\}$ are uniformly bounded in $W^{2,2}$ and $L^1$. Note that this allows the extraction of weakly convergent subsequences -- at least in $W^{2,2}$ -- but that $\ol M^\Xi$ is not weakly closed since for any function $\phi \in \ol M^\Xi$ the translations $\phi_n(x) = \phi(x-n)$ lie in $\ol M^\Xi$ converge to $0$ weakly.

If we invest slightly more work, we can create a function with three bumps -- one high one and two lower ones. If we keep the high bump fixed and send the smaller bumps of to $\pm \infty$, we can create a sequence in $\ol M^\Xi$ whose weak limit is given by only the larger bump and may not lie in $\ol M$. Clearly, this can be done while keeping the maximum and centre of mass of $\phi$ fixed at $0$.
\end{remark}

\subsection{A Problem with Delamination}\label{section delaminating}
The fact that a minimiser $\phi$ of $\E$ is compactly supported suggests that $\gamma$ would attach to the unit circle $\partial B_1(0)$ except on a segment of length proportional to $\delta^{1/3}$.
For a future application, we consider a related problem in which the delamination from the unit circle is penalised: {\em minimise the functional
\[
\com{\W_\alpha:\left(\bigcup_{l>0}\overline \M_l\right)\to [0,\infty),}\qquad \W_\alpha(\gamma) = \W(\gamma) + \alpha\,\H^1\big(S^1\setminus \gamma \big)
\]
in the class $\ol\M_L$.} \com{Repeating the proof of Theorem \ref{theorem short length} for $\W_\alpha$ instead of $\W$, we find the following.}

\begin{theorem}
\com{The energy expansion}
\[
\inf_{\gamma\in \M_{2\pi+\delta}} \com{\W_\alpha(\gamma)}= 2\pi + \Theta_\alpha\,\delta^{1/3} + o(\delta^{1/3})
\]
\com{with the constant}
\[
\Theta_\alpha \com{:}= \inf\left\{\int_\R |\phi''|^2\ds + \alpha\,\H^1\big(\{\phi>0\}\big)\:\bigg|\:\phi\in M\right\}.
\]
\com{holds.}
\end{theorem}

Note that the functional $\E_\alpha(\phi) = \E(\phi) + \alpha\H^1(\{\phi>0\})$ has the same scaling property $\E_\alpha(\phi_\rho) = \rho^{1/3}\E_\alpha(\phi)$ as the original functional $\E$. The variational analysis for this problem is actually easier since there is an a priori bound on the size of the support of $\phi$ in terms of the energy $\E_\alpha$ while no such bound was available purely in terms of the energy $\E = \E_0$.
 \com{A simpler version of Lemma \ref{lemma minimiser} establishes that $\E_\alpha$ has a minimiser $u_\alpha$ in $\ol M$ which satisfies the same Euler-Lagrange equation on $\{u_\alpha>0\}$. The scaling argument of Lemma \ref{lemma minimiser conditions} establishes that the Lagrange multiplier is given in terms of $\E_\alpha$ instead of $\E$, and when determining $\rho$ as in Lemma \ref{lemma minimiser}, one obtains the modified equation}
 \[
\left(\frac{\rho}r\right)^2 = \mu^2 = \frac{\E_\alpha(\phi)}{6\,L(\phi)} = \frac{\left(\frac{\rho^2}{\sin^2\rho} + \frac{\rho}{\sin\rho} -2 + \alpha\right) r}{6\left[-\frac13 + \frac1{2\sin^2\rho} - \frac{\cos\rho}{2\rho\,\sin\rho}\right] r^3}.
\]
\com{The presence of $\alpha$ changes the value of $\rho$, meaning that} the minimiser $\phi_\alpha\in \ol M$ of $\E_\alpha$ is only $C^{1,1} = W^{2,\infty}$- and not $C^2$-smooth. \com{The positivity of $\Theta_\alpha$ is more obvious from the bound on the support and does not need to be established separately, but also follows from $\E_\alpha\geq \E$.} Since the analysis does not simplify as nicely as before, we do not compute minimisers explicitly in this setting, but we obtain scaling results for $\Theta_\alpha$ with $\alpha$. Note that always $\Theta\leq \Theta_\alpha\leq \Theta +4\alpha$ by using the minimiser $\ol u$ of $\E$ as an energy competitor for $\E_\alpha$.

\begin{theorem}
We have
\[
\lim_{\alpha\to \infty} \frac{\Theta_\alpha}{\alpha^{2/3}} = \frac{3\pi^{2/3}}{2}.
\]
\end{theorem}

\begin{proof}
Assume that $\phi \in \ol M$ is supported on an interval $[-r,r]$. Then
\[
1 = \int_{-r}^r \frac{(\phi')^2}2 - \phi\dx \leq \frac12 \int_{-r}^r (\phi')^2\dx \leq \frac12\:\left(\frac{2r}{\pi}\right)^2 \int_{-r}^r |\phi''|^2\dx
\]
with equality if and only if $\phi$ is the $\cos$-shaped transition
\[
\phi(x) = \begin{cases} \frac{(8r)^{1/2}}{\pi}\left(1 + \cos\left(\frac{\pi}{2r}x\right)\right)&|x|<r\\
0&|x|>r\end{cases}.
\]
\com{The correct scaling can be obtained by the standard Poincar\'e inequality for $\phi \in H_0^2(-r,r)$ while the optimal constant is obtained by the same analysis as in Lemma \ref{lemma minimiser} by minimizing the functional $\E$ among functions $\phi \in H^2(-r,r)$ satisfying}
\begin{enumerate}
\item $\phi\geq 0$,
\item $\phi(r) = \phi(-r) = 0$ and
\item $\phi'(-r) = \phi'(r) = 0$,
\item \com{$\int_{-r}^r (\phi')^2\dx = 1$}.
\end{enumerate}
\com{In the absence of linear $\phi$-term, the Euler-Lagrange equation is homogeneous and the transition is cosine-shaped (with a shift to satisfy positivity and boundary conditions as there are only terms of second order or higher).}
This implies that
\[
\W_\alpha(\phi) \geq \frac{\pi^2}{2\,r^2} + \alpha r.
\]
If we optimise this over $r$, we find that
\[
- \frac{\pi^2}{r^3} + \alpha = 0\qquad \Ra\quad r = \left(\frac{\pi^2}\alpha\right)^\frac13 
\]
whence
\[
\W_\alpha(\phi) \geq \frac{\pi^2}2 \left(\frac{\pi^2}\alpha\right)^{-\frac23} + \alpha  \left(\frac{\pi^2}\alpha\right)^\frac13 = \left(\frac{\pi^2\,\pi^{-4/3}}2 + \pi^{2/3}\right)\alpha^{2/3} = \frac{3\,\pi^{2/3}}2\,\alpha^{2/3}
\]
For the opposite inequality, choose any positive bump function $\eta \in C_c^\infty(-1,1)$ such that $\int_{-1}^1 (\eta')^2\dx >2$ and set $\eta_\alpha = \alpha^{-1/6}\eta\left(\alpha^{1/3}x\right)$. Then $\eta_\alpha\in C_c^\infty(-\alpha^{-1/3}, \alpha^{-1/3})$,
\begin{align*}
L(\eta_\alpha) &= \int_{\R} \left(\frac{\left(\alpha^{-\frac16 + \frac13}\eta'\right)^2}2 - \alpha^{-\frac16}\eta\right)(\alpha^{1/3}x)\dx \\
	&= \alpha^{2\left(-\frac16 + \frac13\right) - \frac13} \int_\R\frac{(\eta')^2}2\dx - \alpha^{-\frac16 - \frac13}\int_\R\eta\dx\\ 
	&= \int_\R\frac{(\eta')^2}2\dx - \alpha^{-1/2} \int_\R\eta\dx
\end{align*}
such that $L(\eta_\alpha) \geq 1$ for large enough $\alpha>0$. While $\eta_\alpha\notin M$ for large $\alpha$, we find that $\E_\alpha(\eta_\alpha) \geq \Theta_\alpha$ since $L(\eta_\alpha)\geq 1$ by the same argument as in Corollary \ref{corollary rescaled problem}, so
\begin{align*}
\Theta_\alpha &\leq \E_\alpha (\eta_\alpha)\\
	 &= \int_\R |\eta_\alpha''|^2 \dx + \alpha\cdot2\alpha^{-1/3}\\
	 &= \alpha^{2(-\frac16 + 2\,\frac13) - \frac13}\int_\R|\eta''|^2\dx + \alpha^{2/3}\\
	 &= \left(\int_\R|\eta''|^2\dx + 1\right)\alpha^{2/3}.
\end{align*}
\com{
Choosing $\eta$ to be the optimal cosine-shaped transition function $\phi$ we considered above for the optimal constant in the Poincar\'e-type inequality, we recover the same constant in the upper and lower bound.
}
\end{proof}

\subsection{Application to the Buckling of Cylindrical Shells}\label{section application}

In this section, we derive a simple one-dimensional model for two-layer cylindrical shells and apply Theorem \ref{theorem short length} to get conditions for when the inner layer will buckle away from the outer one in certain scaling regimes. The setting we have in mind is a pipe or tube with an outer layer which contracts more at low temperatures than the inner layer. While the model is simplistic, its benefit is that it provides explicit parameters in terms of universal constants and material properties.
During this section, we remain entirely on a formal level.

The elastic energy of a thin shell can be decomposed into an energy contribution due to stretching which scales asymptotically with the thickness $h$ of the plate and an energy contribution due to bending which scales with $h^3$. For a cylindrical shell, i.e.\ a shell of the form
\[
\Sigma = \big\{(x,y, z)\in \R^3\:|\: (x,y) = \gamma(s) \text{ for }s\in S^1, \: -H < z< H\big\}
\]
the energy can be determined in terms of the planar profile $\gamma$:
\begin{align*}
\E_{stretch} &= c_{stretch}(2H)h\int_{\gamma} \left||\gamma'| - \frac{L}{2\pi}\right|^2\d\H^1\\
\E_{bend} &= h^3\int_\Sigma \chi_H\,\left|\vv H\right|^2 + \chi_K K\d\H^2\\
	&= \frac{\chi_H\,h^3}4\,(2H)\, \int_{S^1} \left|\kappa\right|^2 \d\H^1
\end{align*}
where $\chi_H, \chi_K$ are material parameters, $\vv H$ and $K$ denote the mean curvature vector and Gauss curvature of $\Sigma$ respectively. This bending energy is commonly known as the Helfrich functional and has been derived rigorously as a $\Gamma$-limit of three-dimensional elasticity in \cite{friesecke2002theorem}.
Since $\Sigma$ has a straight direction in $z$, the Gauss curvature of the cylindrical shell vanishes identically, and the mean curvature of $\Sigma$ is the average of the curvature vector $\vv\kappa$ of $\gamma$ and $0$. The stretching energy is minimised for an arc-length parametrised curve, so we consider the normalised elastic energy
\begin{align*}
\E_{el}(\Sigma) 
	&=  c_{stretch}\int_{S^1} \left||\gamma'| - \frac{L}{2\pi}\right|^2\,|\gamma'|\ds + \frac{\chi_H\,h^2}4 \int_{S^1} \kappa^2 \d\H^1\\
	&=   c_{stretch}\left|\frac{\H^1(\gamma)}{2\pi} - \frac{L}{2\pi}\right|^2\,{\H^1(\gamma)} + \frac{\chi_H\,h^2}4 \,\W(\gamma)\\
	&= \frac{ c_{stretch}}{(2\pi)^2}\left(\left|\H^1(\gamma) - L\right|^2\,\H^1(\gamma) + \pi^2\frac{\chi_H}{c_{stretch}}\,h^2\,\W(\gamma)\right)\\
	&\approx \frac{ c_{stretch}\,L }{(2\pi)^2}\left(\left|\H^1(\gamma) - L\right|^2 + \frac{\chi_H}{c_{stretch}L}\,(\pi h)^2\,\W(\gamma)\right)
\end{align*}
which is a purely geometric functional of $\gamma$. We now consider the physical situation of a two-layer cylinder whose layers are composed of materials with different physical properties. We model the layers separately by shells $\Sigma_i, \Sigma_o$ (the inner and the outer layer) given by two planar profiles $\gamma_i,\gamma_o$. If $\H^1(\gamma_i), \H^1(\gamma_o)$ and $\W(\gamma_i), \W(\gamma_o)$ are all approximately $2\pi$, we know that $\gamma_i, \gamma_o$ are $W^{2,2}$-close to the unit circle (up to translation), and in particular that $\gamma_o$ is a Jordan curve that bounds an open set $E_o$. The concepts of inner layer and outer layer now translate to $\gamma_i\subset \ol{E_o}$.

We can model the elastic energy of our two-layer cylinder by
\begin{align*}
\E_{el} &= \E_{el, \gamma_o} + \E_{el,\gamma_i} + \E_{interaction, \gamma_i,\gamma_o}\\
	&= C_o\left(\left|\H^1(\gamma_o) - L_o\right|^2 + \eps_o\,\W(\gamma_o)\right)\\
	&\quad + C_i\left(\left|\H^1(\gamma_i) - L_i\right|^2 + \eps_i\,\W(\gamma_i)\right)\\
	&\quad + \alpha \,\H^1\big(\gamma_o\setminus\gamma_i\big)
\end{align*}
where $C_i, C_o$ are material parameters of the inner and outer shell which model resistance to stretching and $\eps_i, \eps_o$ are parameters which encode resistance to bending and thickness of the shell. The coupling parameter $\alpha\in[0,\infty)$ models that $\Sigma_i, \Sigma_o$ are connected by an adhesive and the contribution to the elastic energy is proportional to the area of delaminating.

In this article, let us consider the asymptotic case $C_o\to \infty$ in which the outer shell is a lot more rigid than the inner one. Then $\gamma_o$ must minimise the energy
\[
C_o\left(\left|\H^1(\gamma_o) - L_o\right|^2 + \eps_o\,\W(\gamma_o)\right).
\]
Since the first part of the functional only depends on the length of $\gamma_o$ and the second part is minimised for a circle, we see that $\gamma_o$ is a circle of radius 
\[
r_o = \argmin_{r>0} \left( \big(r-L_o\big)^2 + \frac{4\pi^2\eps_o}{r}\right) \in \left(L_o, L_o + \frac{2\pi^2\eps_0}{L_o^2}\right)
\] 
since the first term prefers $r$ to be close to $L_o$ and the second term prefers $r$ to be large. The precise constant is determined by verifying that the derivative of the function to be minimised is negative at $L_o$ and positive at $\com{L_0+} \frac{2\pi^2\eps_0}{L_o^2}$. We are then left to find $\gamma_i\subset B_{r_o}(0)$ such that $\gamma_i$ minimises 
\[
C_i\left[\left|\H^1(\gamma_i) - L_i\right|^2 + \eps_i\,\W(\gamma_i)\right] + \alpha \,\H^1\big(\gamma_o\setminus \gamma_i\big).
\]
The interesting case for us is when $L_i> r_o$, i.e.\ when the outer shell contracts more for low temperature than the inner shell. The inner shell has three options:
\begin{enumerate}
\item compression,
\item buckling,
\item fracture.
\end{enumerate}
Assuming that fracture does not occur, we try to distinguish whether buckling or compression is energetically favourable, always assuming that all shells remain cylindrical. This corresponds to buckling by ridge formation, while blistering is excluded from this analysis. This assumption is reasonable since an initially Gauss-flat cylindrical shell wants to remain Gauss-flat due to the non-stretching (isometry) constraint, which suggests a cylindrical shape (assuming high enough regularity for curvature arguments). 

Denote $L_i = r_o + \delta$, $\H^1(\gamma_i) = r_o + t$ for some $t\geq 0$. Then, from the previous analysis, we see that to leading order we have
\begin{align*}
C_i\left[\big|\H^1(\gamma_i) - L_i\big|^2 + \eps_i \W(\gamma_i)\right] + \alpha\,\H^1(\gamma) &= C_i\left[\big|\H^1(\gamma_i) - L_i\big|^2 + \eps_i \left(\W(\gamma_i) + \frac{\alpha}{C_i\eps_i}\H^1(S^1\setminus \gamma)\right)\right]\\
	&\approx C_i\left[ (t-\delta)^2 + \eps_i \left(\frac{4\pi^2}{r_o} + \frac{\Theta_{\tilde\alpha}}{r_o^{4/3}}\,t^{1/3}\right)\right]\\
	&= C_i\left[\frac{4\pi^2\,\eps_i}{r_o} + (t-\delta)^2 + \frac{\Theta_{\tilde\alpha}\,\eps_i}{r_o^{4/3}}\,t^{1/3}\right]\\
	&= C_i\left[\frac{4\pi^2\,\eps_i}{r_o} + \delta^2 \left[ \left(\frac t\delta -1 \right)^2 + \frac{\Theta_{\tilde\alpha}\,\eps_i}{r_o^{4/3}\delta^{5/3}}\,\left(\frac t\delta\right)^{1/3}\right]\right]
\end{align*}
as we introduce the scaled parameter $\tilde\alpha = \frac{\alpha}{C_i\eps_i}$. The question whether buckling or compression is energetically favourable thus reduces to the question whether the function
\[
e_\lambda(s) = (s-1)^2 + \lambda s^{1/3}
\]
has its minimum at $0$ or a positive number $s = \frac t\delta$ given the parameter $\lambda = \frac{\Theta_{\tilde\alpha}\,\eps_i}{r_o^{4/3}\delta^{5/3}}$. We observe the following:
\begin{enumerate}
\item $e_\lambda(0) = 1$ for all $\lambda>0$,
\item $\min_{s\in\R} e_\lambda(s)$ is increasing in $\lambda$ and
\item $e_\lambda(0) = e_\lambda(1) = 1$ if $\lambda =1$, so buckling is favourable for $\lambda\leq 1$ (because $e_\lambda'(1) = 2\cdot (1-1) + \frac13 >0$, $e_\lambda$ does not assume its minimum at $1$, which makes the minimum lower).
\end{enumerate}
\begin{figure}
\begin{tikzpicture}[yscale=2, xscale=4]
\node at (-0.1, -0.1){};
\node at (1.201, 1.201){};

\draw[->] (-0.03 ,0) -- (1.2,0);
\draw[->] (0,-0.03) -- (0,2.1);

\draw[dashed](0,1) -- (1.2,1);

\draw[blue, domain = 0:1.2, samples =200] plot  ( {\x}, {(\x-1)^2});
\draw[purple, domain = 0:1.2, samples =200] plot  ( {\x}, {(\x-1)^2 + 0.5*\x^(1/3)});
\draw[red, domain = 0:1.2, samples =300] plot  ( {\x}, {(\x-1)^2 + 1*\x^(1/3)});
\draw[brown, domain = 0:1.2, samples =300] plot  ( {\x}, {(\x-1)^2 + 1.05*\x^(1/3)});
\draw[yellow, domain = 0:1.2, samples =300] plot  ( {\x}, {(\x-1)^2 + 1.5*\x^(1/3)});
\draw[green, domain = 0:1.2, samples =300] plot  ( {\x}, {(\x-1)^2 + 2*\x^(1/3)});

\foreach \x in {0, 0.5, ..., 2} 
	{
	\draw[]( -0.03,\x) -- ++ (0.06,0);
	\node [left] at (0,\x) {\x};
	}
\foreach \x in {0, 0.25, ..., 1} 
	{
	\draw[](\x, -0.03) -- ++ (0, 0.06);
	\node [below] at (\x,0) {\x};
	}
\end{tikzpicture}
\caption{\label{figure scaling}The function $(s-1)^2 + \lambda s^{1/3}$ for values $\lambda=0$ (blue), $\lambda=1/2$ (purple), $\lambda = 1$ (red), $\lambda=1.05$ (brown), $\lambda=1.5$ (yellow) and $\lambda=2$ (green).}
\end{figure}
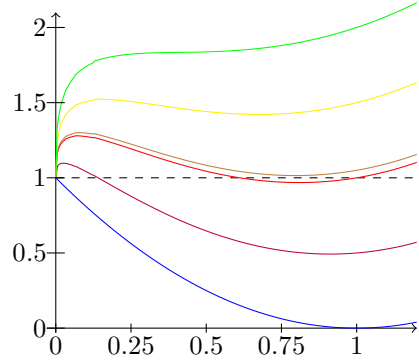
However, we see that
\[
e_\lambda'(s) = 2(s-1) + \frac{\lambda}3 \,s^{-2/3} 
\]
satisfies $\lim_{s\to 0}e_\lambda'(s) = \lim_{s\to \infty}e_\lambda'(s) = \infty$, so $e_\lambda'$ assumes a minimum at a point $s\in (0,\infty)$ where 
\[
0 = e_\lambda''(s) = 2 - \frac{2\lambda}9\,s^{-5/3} \qquad\Ra\quad s = \left(\frac\lambda 9\right)^\frac 35
\]
which means that
\[
e_\lambda'(s) = 2s + \frac{\lambda}3 \,s^{-2/3} -2 = 2\left(\frac\lambda 9\right)^\frac 35  + 3\frac{\lambda}9 \left(\frac\lambda 9\right)^{-\frac25} -2 = 5\left(\frac\lambda 9\right)^\frac 35 -2,
\]
so if $\lambda> 9 \left(\frac 25\right)^{\com{5/3}}\approx \com{1.95}$, we have $e_\lambda'>0$ on $(0,\infty)$ and the minimum is zero. We therefore find that there must exist a critical threshold $\lambda_0\in (1, 5.2)$ such that the global minimum of $e_\lambda$ is assumed at a positive values $s$ for $\lambda<\lambda_0$ and at $0$ for $\lambda>\lambda_0$. Numerically, we find that $\lambda_0 \in (1.0341,1.0342)$ (see also Figure \ref{figure scaling}). We thus expect bifurcation to \com{compression if}
\[
\lambda = \frac{\Theta_{\tilde\alpha}\,\eps_i}{r_o^{4/3}\delta^{5/3}} < \lambda_0\qquad\LRa\quad \delta^{5/3} < \frac1{\lambda_0} \frac{\Theta_{\tilde\alpha}\,\eps_i}{r_o^{4/3}}\qquad\LRa\quad \delta < \left(\frac1{\lambda_0} \frac{\Theta_{\tilde\alpha}\,\eps_i}{r_o^{4/3}}\right)^\frac35
\]
and buckling if the strict opposite inequality holds.
Recall that, if $\alpha = 0$ we have $\Theta_{\tilde\alpha} = \Theta$, $\eps_i = \frac{\chi_H\,\pi^2}{c_{stretch}L}h^2$ and thus we expect to see bifurcation to buckling if the preferred excess length $\delta$ satisfies
\[
\delta = \left(\frac1{\lambda_0} \frac{\Theta\,\pi^2\chi_H}{c_{stretch}\,(r_o+\delta)\,r_o^{4/3}}\right)^\frac35h^\frac65 \approx \left(\frac{\Theta\,\pi^2\chi_H}{\lambda_0\,c_{stretch}\,r_o^{7/3}}\right)^\frac35 h^\frac65.
\]
In the presence of an adhesive, we see that $\tilde\alpha = \frac{\alpha}{C_i\eps_i} = \frac{4\alpha} {\chi_H\,h^2}$ which is large for small $h$, so 
\[
\Theta_{\tilde\alpha} \approx \frac{4\pi^{2/3}}2\,\tilde\alpha^{2/3} = \frac{3}2\left(\frac{4\pi \,\alpha}{\chi_H}\right)^{2/3} h^{-4/3}
\]
such that we expect bifurcation to buckling at a preferred excess length
\begin{align*}
\delta & = \left(\frac1{\lambda_0} \frac{\Theta_{\tilde\alpha}\,\pi^2\chi_H\,h^2}{c_{stretch}\,(r_o+\delta)\,r_o^{4/3}}\right)^{3/5}
	\quad\approx \left(\frac{4\pi\alpha}{\chi_H}\right)^{2/5} \left(\frac{3}{2\lambda_0}\frac{\pi^2\chi_H\,h^{2/3}}{c_{stretch}r_o^{7/3}}\right)^{3/5} \\
	&= \frac{4^{2/5}\,3^{2/5}\,\pi^{8/5}}{(2\lambda_0)^{2/5}} \: \frac{\chi_H^{1/5}}{c_{stretch}^{3/5}r_o^{7/5}}\:(\alpha h)^{2/5}
\end{align*}
with compression below this threshold and possible buckling above the threshold. In the setting of strong adhesion $\alpha\sim h^{-1}$, we leave the regime of small $\delta$ and the asymptotic analysis becomes invalid.

Note that we assumed $\gamma$ to be arc-length parametrised after buckling. This is a sensible assumption in the case $\alpha=0$ where tangential slip along the exterior shell is possible, but an over-simplification in the presence of an adhesive, meaning that the buckling cost would be higher than assumed here. We recall, however, that the first derivatives of the buckling profile decay as $\delta^{1/3}$ in $L^\infty$ and as $\delta^{1/2}$ in $L^2$, which suggests that the stretching effect along the buckling profile should not influence the total energy to leading order. 

In the other asymptotic regime $C_i\to \infty$, the {\em inner} shell is given by a circular profile and the outer shell attaches to the inner shell everywhere, with or without adhesive. Whether the outer profile remains a circle in the true competitive regime $1 < C_i, C_o \ll \infty$ remains open.

\section{The Large Length Limit} \label{section large length}

Lemma \ref{lemma muller-roger} implies that $\W(\gamma) \geq \H^1(\gamma)$ for all curves $\gamma\subset \overline{B_1(0)}$ with equality if and only if $\gamma$ is a (multiply covered) circle of radius $1$. In particular, if $\H^1(\gamma)\neq 2\pi$, we have $\W(\gamma) > \H^1(\gamma)$ since a multiply covered circle cannot be approximated by embedded curves. In this section, we construct a family of curves $\gamma_L$ of length $L$ for sufficiently large $L$ such that 
\[
\lim_{L\to \infty} \frac{\W(\gamma_L)}L = 1,
\]
recovering the optimal scaling to leading order. We show slightly more, namely that
\[
\limsup_{L\to \infty} \frac{ \W(\gamma_L) - L}{L^{1/2}} <\infty,
\]
but do not characterise the first order term more precisely. The idea of constructing $\gamma_L$ is as follows:

\begin{enumerate}
\item The elastic energy of a multiply covered circle of radius $\rho_L$ (not necessarily a closed curve) is $L\,\rho_L^{-2}$.

\item While a multiply covered circle cannot be approximated by embedded curves with bounded energy, we can approximate a multiply covered circle with end tied of in two loops by spiralling curves with two loops -- see Figure \ref{figure long length}. 

\item The energy of the inner loop-like appendage is just roughly constant in $L$, whereas the energy of the outer appendage is inversely proportional to the space between the spiral and the domain boundary $1- \rho_L$. We approximate the energy of the spiral by
\[
L\cdot \frac1{\rho_L^2} = L \cdot \frac{1}{(1 - (\rho_L-1))^2} \approx L\cdot \left(1 + 3\,(1-\rho_L)\right) = L + 3L\,(1-\rho_L).
\] 
Trying to match the orders of the leading excess energy terms, we have to satisfy
\[
\frac1{1-\rho_L} \sim \,(1-\rho_L)\com{L}\qquad \Ra \quad (1-\rho_L)^2 \sim  \frac1L\qquad\Ra\quad \rho_L \approx 1- cL^{-1/2}.
\]
for a suitable constant $c$.
\end{enumerate}

The optimal $c$ would have to be found by optimising over the shape of loops and balancing the terms. We do not see applications for this fine structure and do not execute this step.

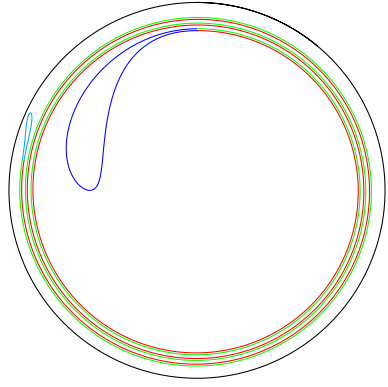
\begin{figure}
\begin{center}
\begin{tikzpicture}[scale=2.5]
\node at (1.001, 1.001){};
\node at (-1.001, -1.001){};

\draw[domain = 0:400, samples = 400] plot ({sin(\x)}, {cos(\x)});
\draw[domain = 0 :1000, samples = 1000, green] plot ({(0.86 + 0.03*\x/360)*sin(\x)}, {(0.86 + 0.03*\x/360)*cos(\x)} );
\draw[domain = -1:1.003, samples =300, blue] plot ( { -(0.855 - 0.005*\x)* (1 - (1-\x) *(1+\x)*(1+\x)/3) *cos (90*\x*\x)}, {(0.855 - 0.005*\x) *(1 - (1-\x) *(1-\x) *(1+\x)*(1+\x)/3) *sin (\x*\x * 90)} );
\draw[domain = -1:1, samples =60, cyan] plot ( { (0.855 +0.03*1000/360 + 0.005*\x)* (1 + (1-\x) *(1-\x) * (1+\x)*(1+\x)*(1+\x)/25) *sin (-15*\x*\x+1015)}, {(0.855 +0.03*1000/360 + 0.005*\x) *(1 + (1-\x) *(1-\x) *(1+\x)*(1+\x)*(1+\x)/25) *cos (-\x*\x * 15 + 1015)} );
\draw[domain = 0 :1000, samples = 1000, red] plot ({(0.85 + 0.03*\x/360)*sin(\x)}, {(0.85 + 0.03*\x/360)*cos(\x)} );
\end{tikzpicture}
\end{center}
\caption{\label{figure long length}The heuristic construction for an energy optimal sequence with large length: spirals (red and green), inner loop (blue) and outer loop (cyan).}
\end{figure}

\begin{proof}[Proof of Theorem \ref{theorem large length}]
Of course, it suffices to construct curves $\gamma_L\in\M_L$ such that $\W(\gamma_L) \leq L + C\,\left(\sqrt L + 1\right)$. Make the following ansatz: $\gamma = \gamma_1\oplus \gamma_2 \oplus \gamma_3\oplus \gamma_4$ where
\begin{align*}
\gamma_1:[0,\ell]\to \R^2, &\qquad \gamma_1(s) = \left(\rho_L + \sigma_Lf\left(\frac s\ell \right)\right) \begin{pmatrix} \cos s\\ \sin s\end{pmatrix} \\
\gamma_3:[0,\ell]\to \R^2, &\qquad \gamma_{\com{3}}(s) = \left(\rho_L + \eps_L +  \sigma_Lf\left(\frac{\ell -s}\ell \right)\right) \begin{pmatrix} \cos (\ell-s)\\ \sin (\ell-s) \end{pmatrix},
\end{align*}
$0< \rho_L < 1$, $0 < \sigma_L < \frac{1-\rho_L}2$, $f\in C^\infty([0,1], [0,1])$ is a strictly monotone increasing function satisfying
\[
f(0) = 0, \qquad f(1) = 1, \qquad f^{(k)}(0) = f^{(k)}(1) = 0\quad\forall\ k\geq 1
\]
\com{
and $0 < \delta_L < \frac{1- \rho_L - \sigma_L}2$ is so small that
\[
f\left(s + \frac{2\pi}\ell\right) > f(s) + \frac{\eps_L}{\sigma_L}\qquad\forall\ s\in \left[0,1-\frac{2\pi}\ell\right].
\]
}
 By construction, the curves $\gamma_1,\gamma_3$ are embedded and do not touch. \com{To see this, note the following.
 \begin{enumerate}
 \item The angular velocity of $\gamma_1, \gamma_3$ is always non-zero since the radial function does not vanish, so both curves are immersed. 
 \item If $\gamma_1(s) = \gamma_1(s')$, then $|\gamma_1(s)| = |\gamma_1(s')|$ and thus $s= s'$ since $f$ is strictly monotone, and the same for $\gamma_3$. Since $\gamma_1,\gamma_3$ are defined on a compact interval, this suffices to show embeddedness.
 \item If $\gamma_1(s) = \gamma_3(s')$, then $s - \ell - s'\in 2\pi\Z$ and
 \[
 f\left(\frac s\ell\right) = \frac{\eps_L}{\sigma_L} + f\left(\frac {\ell - s'}\ell\right).
 \]
 This impossible due to the choice of $\eps_L$.
 \end{enumerate}
 Finally, we take a curve $\tilde\gamma_2:[-1,1]\to \R^2$ for the loops. 
 We choose some smooth function $g:[-1,1]\to [1/2,1]$ that satisfies
 \[
 g(-1)= g(1) = 1, \qquad g^{(k)}(-1) = g^{(k)}(1) = 0 \qquad \forall\ k\geq 1, \qquad g(s) < g(-s)\quad\forall\ s\in (0,1)
 \] 
 and set
 \[
 \tilde \gamma_2:[-1,1]\to \overline{B_1(0)}, \qquad \tilde \gamma_2(s) =  g(s)\,\begin{pmatrix}\cos\big(1-f(|s|)\big)\\ \sin\big(1-f(|s|)\big)\end{pmatrix},\qquad \gamma_2(s) = \left(\rho + \eps_L f\left(s\right)\right)\,\tilde\gamma_2(2s+1)
 \]
 for the inner loop.
 }
The outer loop can be constructed similarly, using a translation and a rotation $Q$
\[
\gamma_4(s) = v_0 + Q\left( \left(1- {\eps_L}\,f\left(\frac{s+1}2\right)\right)\left(1-\rho_L +\frac{\sigma_L}2\right)\,\tilde\gamma_2(s)\right).
\]
Finally, we can take the limit $\sigma_L, \eps_L\to 0$ under which $\gamma_1, \gamma_3$ approach a multiple cover of a circle of radius $1-\rho_L$, $\gamma_2$ approaches $(1-\rho_L)\tilde\gamma_2$ and $\gamma_4$ approaches $v_0 + \rho_L\,O\tilde\gamma_4$ such that
\begin{align*}
\lim_{\sigma_L,\eps_L\to 0} \H^1(\gamma) &= 2\ell\rho_L + \rho_L\,\H^1(\tilde \gamma_2) + (1-\rho_L)\,\H^1(\tilde\gamma_2)\\
\lim_{\sigma_L,\eps_L\to 0} \W(\gamma) &= \frac{2\ell}{\rho_L^2} + \frac{\W(\tilde\gamma_2)}{\rho_L} + \frac{\W(\tilde\gamma_2)}{1-\rho_L}\\
	&\approx 2\ell + 3\,\ell(1-\rho_L) + \frac{\W(\tilde\gamma_2)}{\rho_L} + \W(\tilde \gamma_2) + \frac{\W(\tilde\gamma_2)}{1-\rho_L}
\end{align*}
which requires us to choose $\ell = \frac L2 + O(\rho_L^{-1})$ to match the length constraint. To balance the orders $(1-\rho_L)^{-1}, \ell(1-\rho_L)$ in the error term, we need to choose $1-\rho_L = O(L^{-1/2})$. \com{Then
\[
\H^1(\gamma) = L, \qquad \W(\gamma)= L + \big(3 + \W(\tilde\gamma_2)\big)\,O(\sqrt{L}) + O(1).
\]
}
\end{proof}

\section{Curves in Three Dimensions}\label{section 3d}

Finally we show that, to an extent, the phenomena described above were two-dimensional and can be avoided if curves are permitted to buckle out of plane.

\begin{proof}[Proof of Theorem \ref{theorem 3d}]
In the case of small excess length $L=2\pi + \delta$ for small $\delta$, it suffices to show that there exists a smooth curve $\gamma_\delta$ embedded into $B_1(0)$ such that $\H^1(\gamma_\delta) = 2\pi + \delta$, $\W(\gamma_\delta)\leq 2\pi + C\,\delta$. As a competitor, consider the curve 
\[
\gamma_\eta(s) = \sqrt{1-\eta^2}\begin{pmatrix}\cos s\\ \sin s\\0\end{pmatrix} + \frac\eta{\sqrt 2} \begin{pmatrix}0\\0\\ \cos(m s)\end{pmatrix}
\]
which satisfies $|\gamma_{\com\eta}|^2 =1-\eta^2 + \frac{\eta^2}2 \cos^2(ms) < 1$ and
\begin{align*}
\H^1(\gamma_\eta) &= \int_0^{2\pi}\sqrt{1-\eta^2 + \frac{m^2\eta^2}{2} \cos^2(ms)}\ds\\
	&= \int_0^{2\pi} 1+ \com{\frac12}\left(\frac{m^2\cos^2(ms)}{2}-1\right)\eta^2 + O(\eta^4)\ds\\
	&= 2\pi + \com{\frac12}\left(\frac m2\int_0^{2\pi}\cos^2(ms)\,m\ds - 2\pi\right)\eta^2 + O(\eta^4)\\
	&= 2\pi + \com{\frac12}\left(\frac m2 \int_0^{2m\pi}\cos^2(s')\ds'-2\pi\right)\eta^2 + O(\eta^4)\\
	&= 2\pi + \com{\frac12}\left(\frac m2 \cdot {m\pi} - 2\pi\right)\eta^2 + O(\eta^4)\\
	&= 2\pi + \left(\frac{m^2}{\com 4} - \com{1}\right)\pi\,\eta^2 + O(\eta^4),
\end{align*}
so if we choose $m=\com{3}$, we find that $\H^1(\gamma_\eta) = 2\pi + \com{\frac54}\pi\eta^2 + O(\eta^4)$ so we can choose $\eta = \sqrt{\frac\delta{\com{\frac54}\pi}} + o(\delta)$ such that $\H^1(\gamma_\eta) = 2\pi + \delta$. We further find by the same calculation that
\[
\int_0^{2\pi}|\gamma_\eta''|^2\ds  = 2\pi + \com{\left(\frac{m^4}2 -2\right)\pi\,\eta^2 + O(\eta^4) = 2\pi + \frac{77}2}\,\pi \eta^2 + O(\eta^4)
\]
which coincides with the elastic energy up to leading order, by much the same calculation as before, so
\[
\W(\gamma_\delta) = 2\pi + \com{\frac{\frac{77}2\,\pi}{\frac54\,\pi}}\delta + O(\delta^2).
\]
Finally, for the long length limit, note that we can use the energy competitor  $\gamma = \gamma_1\oplus \gamma_2 \oplus \gamma_3\oplus \gamma_4$ where $\gamma_1, \gamma_2, \gamma_3$ are as before in the $\{x_3=0\}$ coordinate plane and we set
\[
\gamma_4(s) = \left(\rho_L + \frac{\sigma_L}2 \big(1+f(s)\big)\right)\,\tilde \gamma_2(s) + \eps_L\,h(s)\,e_3
\]
where $g\in C^\infty(\R)$ satisfies $h(s) = 0$ for $s\leq 0$ and $s\geq 1$, but $h(s)>0$ for $s\in (0,1)$. The curve is embedded into the unit ball if we choose $\eps_L$ small enough. This time, we can also take $\rho_L\to 1$ since there is no longer an energy contribution proportional to $(1-\rho_L)^{-1}$. The limiting length is 
\[
\lim_{\rho_L, \eps_L\to 0} \H^1(\gamma) = 2\left(\ell + \H^1(\tilde \gamma_2)\right), \qquad
\lim_{\rho_L, \eps_L\to 0} \H^1(\gamma) = 2\left(\ell + \W(\tilde \gamma_2)\right)
\]
and we just have to choose $\ell = L - \H^1(\tilde\gamma_2)$ to match the constraints.
\end{proof}

\section{Conclusion}\label{conclusion}

We can identify four parameter regimes for the following problem: {\em Minimise $\W(\gamma) = \int_{\gamma}\kappa^2\d\H^1$ among all curves of length $L$ which are embedded into the two-dimensional unit disc.} In all regimes, minimisers exist and are non-unique, with the sole exception of $L=2\pi$ where the unit circle is the unique minimiser.

\begin{enumerate}
\item $L\leq 2\pi$. In this regime, minimisers are given by circles of length $L$ and have energy $\W = \frac{4\pi^2}{L}$. In particular, the energy is decreasing with increasing length. \com{M}inimisers for $L<2\pi$ have a translational degree of non-uniqueness. The same is true for curves in higher dimensions, where another rotational degree of non-uniqueness is introduced.

\item $\com{2\pi <}L <2\pi + \delta_0$ for some sufficiently small $\delta_0>0$. In this regime, the minimum energy scales like 
\[
\min_{|\gamma|=L} \W(\gamma) = 2\pi + \Theta(L-2\pi)^{1/3} + o(L-2\pi)^{1/3}
\]
where 
\[
\Theta:= \inf\left\{\int_\R|\phi''|^2\ds\:|\:\int_{\R}\frac{(\phi')^2}2-\phi \ds = 1\right\} \approx 37.
\]
We expect minimisers to be shaped like minimisers of the associated problem on the real line (see Figure \ref{figure minimiser}) in radial direction, suitably rescaled. In particular, we expect them to attach to the circle except on a segment of length $\sim (L-2\pi)^{1/3}$ where they form a single bump of height $\sim (L-2\pi)^{2/3}$. While we have not proved that the energy increases with increasing length, the highest order term does and we expect the statement to be true, qualitatively differing from the previous regime.

Minimisers cannot be circles (or unions of circles) in this regime and thus do not posses radial symmetry. However, the set of minimisers is rotationally symmetric, so minimisers cannot be unique. We have thus entered a truly non-linear regime. Since the solution of the associated minimisation problem on the real line is symmetric, we expect the rotational symmetry to be replaced by at least a reflectional symmetry for the individual minimisers.

For curves in three dimensions, the infimum energy scales as $\inf_{|\gamma|=2\pi + \delta}\W(\gamma)-2\pi\sim \delta$ instead and we observe out-of-plane buckling. Again, up to \com{leading} order the energy is increasing with increasing length. Minimisers are almost planar (since they are $C^1$-close to a circle.

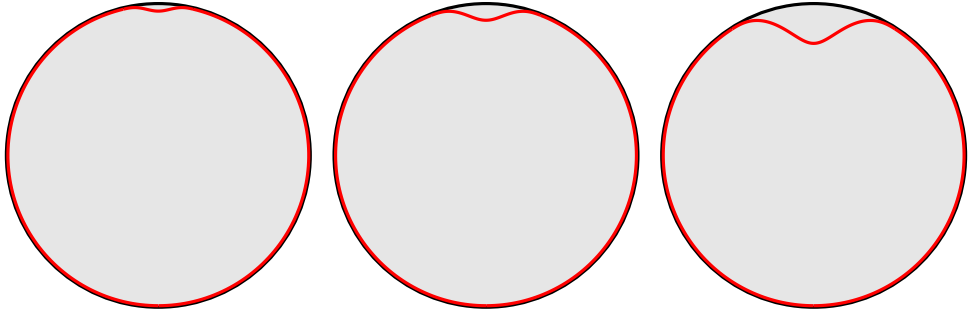
\begin{figure}
\begin{tikzpicture}[scale=2]
\node at (-1.001, -1.001){};
\node at (1.001, 1.001){};

\fill [opacity=0.1] (0,0) circle (1cm);
\draw [very thick=0.2] (0,0) circle (1.01cm);

\draw[very thick, red, domain = -0.2339 : 0.2339, samples = 200] plot ( {(1 - 512^(-2/3)* (1.8145 - 0.5*512^(2/3)*\x*\x + 0.7528 * cos (512^(1/3)* 180/pi*2.4728 *\x))) *sin(\x*180/pi)}, {(1- 512^(-2/3)* (1.8145 - 0.5*512^(2/3)*\x*\x + 0.7528 * cos (512^(1/3)* 180/pi*2.4728 *\x))) *cos(\x*180/pi)});
\draw[very thick, red, domain =  0.2339 : pi, samples = 100] plot ({sin(\x * 180/pi)}, {cos(\x*180/pi)});
\draw[very thick, red, domain =  -pi : -0.2339, samples = 100] plot ({sin(\x * 180/pi)}, {cos(\x*180/pi)});

\end{tikzpicture}
\begin{tikzpicture}[scale=2]
\node at (-1.001, -1.001){};
\node at (1.001, 1.001){};

\fill [opacity=0.1] (0,0) circle (1cm);
\draw [very thick=0.2] (0,0) circle (1.01cm);

\draw[very thick, red, domain = -0.3712 : 0.3712, samples = 200] plot ( {(1 - 128^(-2/3)* (1.8145 - 0.5*128^(2/3)*\x*\x + 0.7528 * cos (128^(1/3)* 180/pi*2.4728 *\x))) *sin(\x*180/pi)}, {(1- 128^(-2/3)* (1.8145 - 0.5*128^(2/3)*\x*\x + 0.7528 * cos (128^(1/3)* 180/pi*2.4728 *\x))) *cos(\x*180/pi)});
\draw[very thick, red, domain =  0.3712 : pi, samples = 100] plot ({sin(\x * 180/pi)}, {cos(\x*180/pi)});
\draw[very thick, red, domain =  -pi : -0.3712, samples = 100] plot ({sin(\x * 180/pi)}, {cos(\x*180/pi)});
\end{tikzpicture}
\begin{tikzpicture}[scale=2]
\node at (-1.001, -1.001){};
\node at (1.001, 1.001){};

\fill [opacity=0.1] (0,0) circle (1cm);
\draw [very thick=0.2] (0,0) circle (1.01cm);

\draw[very thick, red, domain = -0.5724 : 0.5724, samples = 200] plot ( {(1 - 32^(-2/3)* (1.8145 - 0.5*32^(2/3)*\x*\x + 0.7528 * cos (32^(1/3)* 180/pi*2.4728 *\x))) *sin(\x*180/pi)}, {(1- 32^(-2/3)* (1.8145 - 0.5*32^(2/3)*\x*\x + 0.7528 * cos (32^(1/3)* 180/pi*2.4728 *\x))) *cos(\x*180/pi)});
\draw[very thick, red, domain =  0.5724 : pi, samples = 100] plot ({sin(\x * 180/pi)}, {cos(\x*180/pi)});
\draw[very thick, red, domain =  -pi : -0.5724, samples = 100] plot ({sin(\x * 180/pi)}, {cos(\x*180/pi)});

\end{tikzpicture}

\caption{Our approximation of a minimiser of excess length $\delta = 2^{-9}$ (left), $\delta = 2^{-7}$ (middle) and $\delta = 2^{-5}$ (right).}
\end{figure}

\item $2\pi+\delta_0 < L \ll \infty$. In this regime, we have no results. We expect a higher degree of stability here under small changes of length. If $L$ is small enough, a minimiser must touch the boundary of the unit disk but cannot have points of higher multiplicity. As $L$ increases beyond a second threshold, a minimiser must touch the boundary or have points of self-contact (possibly both). We conjecture that any minimiser also in this regime touches the boundary and that the energy is increasing with increasing length.

This regime seems more amenable to numerical computations, utilising for example phase-field methods developed in \cite{dondl:2011eh} and improved in \cite{MR3590663}. This is the only regime in which truly three-dimensional curves may appear as energy minimisers.

\item $L\to \infty$. In this regime, the energy minimum scales linearly with $L$ and the remainder term is bounded by $O(\sqrt L)$ in two dimensions, $O(1)$ in higher dimensions. Minimising curves in two dimensions are expected to be spiralling approximations of a multiply covered circle of radius $1-C\,L^{-1/2}$ with two loops for closedness, one large and inside the interior circle, the other small and between the approximated circle and the domain boundary \com(see Figure \ref{figure long length}). In higher dimensions, we expect the same planar spiralling profile in the limit, except that the exterior loop can be brought into the circle by out-of-plane buckling.

Again, since minimisers cannot be rotationally symmetric, they cannot be unique.
\end{enumerate}

While our analysis was for curves in the unit disk, a simple scaling argument extends our results to disks of any radius. It stands to conjecture that attaching to the boundary would remain optimal in convex domains $C^2$-close to a disk and that in domains with non-constant boundary curvature buckling should happen at the point of {\em lowest boundary curvature} since the constant $\Theta_R$ of energy increase for small excess length in disks $B_R(0)$ is given by $\frac{\Theta}{R^{4/3}} = \Theta\,\kappa^{4/3}$ so that the prefactor decreases rapidly with decreasing curvature.

If curves are allowed to be slightly compressible, we expect to see either buckling away from the boundary or compression along the boundary, depending on the competition between the stretching and bending energy contributions and the amount of excess length. Without adhesion between the boundary and the curve, we expect to see bifurcation to buckling as a curve's preferred excess length exceeds
\[
\delta_{crit} = \left(\frac{\Theta\,\pi^2\chi_H}{\lambda_0\,c_{stretch}\,r_o^{7/3}}\right)^\frac35 h^\frac65 \approx 33.62\, \frac{\chi_H^{3/5}}{c_{stretch}^{3/5}\,r_o^{7/5}}\,h^{6/5}
\]
where $\chi_H, c_{stretch}$ are material parameters, $r_o$ is the radius of the disk that the curve is confined to, $\Theta$ is as above and $\lambda_0$ is an explicit parameter, and $h$ is the thickness of a membrane modelled by the curve (or the diameter of the cross-section of a rod modelled by $\gamma$). If an adhesion $\alpha>0$ is included and the functional $\alpha$ is considered instead, we expect bifurcation to buckling as $\delta$ exceeds
\[
\delta_{crit} = \frac{4^{2/5}\,3^{2/5}\,\pi^{8/5}}{(2\lambda_0)^{2/5}} \: \frac{\chi_H^{1/5}}{c_{stretch}^{3/5}r_o^{7/5}}\:(\alpha h)^{2/5} \approx 12.61 \,\frac{\chi_H^{1/5}}{c_{stretch}^{3/5}r_o^{7/5}}\:(\alpha h)^{2/5}
\]
where $\alpha$ models the strength of the adhesion.

\section*{Acknowledgements} 

The author would like to thank Patrick Dondl and Matthias R\"oger for drawing his attention to the subject and the anonymous referees for their valuable feedback.

\appendix

\section{Proofs of Basic Properties}

Let us begin by proving the basic scaling and energy estimates.

\begin{proof}[Proof of Lemma \ref{lemma basics}]
{\bf First claim.} \com{Immediate from scaling properties.}

{\bf Second claim.} \com{See proof of fourth claim. For embedded curves, it also follows from Hopf's {\em Umlaufsatz}
\[
(2\pi)^2 = \left(\int_\gamma\kappa\d\H^1\right)^2 \leq \H^1(\gamma)\,\W(\gamma).
\]
}

{\bf Third claim.} This is essentially a convenient way to phrase the second claim.

{\bf Fourth claim.} Without loss of generality, we assume that $t_1 = 0$. We can now decompose the curve $\gamma$ into segments $\gamma^j = \gamma|_{[t_j, t_{j+1}]}$ where we identify $L = t_{k+1} = t_0$ modulo $L$ -- note that $\gamma^j$ is a curve segment, not a coordinate function. Then, applying Poincar\'e's inequality to the coordinate functions $\gamma^j_i$, we find that  
\[
\int_{t_j}^{t_{j+1}} |(\gamma^j_i)'|^2 \ds \leq \frac{|t_{j+1}- t_j|^2}{\pi^2} \int_{t_j}^{t_{j+1}} |(\gamma^j_i)''|^2\ds.
\]
\com{
Equality is attained for eigenfunctions of the Laplacian. We note that the boundary conditions for $\gamma^j$ are $\gamma^j(t_j) = \gamma^j(t_{j+1})$, so $\int_{t_j}^{t_{j+1}}(\gamma_i^j)'\ds = 0$. This only leaves the option that $\gamma^j_i(s) = \alpha_{ij}\,\sin\left(\frac{\pi (s-t_j)}{t_{j+1}-t_j}\right)$ for all $j=1,\dots,k$ and $i= 1,\dots, n$. At this point, we make two observations.
\begin{enumerate}
\item In the situation of the first claim, the curve $t_2 = L$ (i.e.\ $\gamma$ may not have any multiple points), and the curve is even $C^1$-periodic. That excludes the first eigenfunction and allows only
\[
\gamma_i(s) = \alpha_{i}\,\cos\left(\frac{2\pi s}{L}\right)\quad\text{or }\gamma_i(s) = \alpha_{i}\,\sin\left(\frac{2\pi s}{L}\right)
\]
with the larger constant $4\pi^2$ which establishes the first claim.
\item A curve $\gamma$ such that all coordinate functions are multiples of the $\sin$ of the same argument cannot be parametrized by arc-length. Hence we conclude that in fact
\[
\sum_{i=1}^n \int_{t_j}^{t_{j+1}} |(\gamma^j_i)'|^2 \ds < \sum_{i=1}^n \frac{|t_{j+1}- t_j|^2}{\pi^2} \int_{t_j}^{t_{j+1}} |(\gamma^j_i)''|^2\ds
\]
for all admissible curves. We conclude that
\[
C:= \inf_{\gamma(0) = \gamma(1) = 0} \W(\gamma) > \frac{\pi^2}{\H^1(\gamma)}.
\]
\end{enumerate}
}
From this we obtain
\begin{align*}
\W(\gamma) &= \sum_{j=1}^k \int_{t_j}^{t_{j+1}} \big|(\gamma)''\big|^2\ds\\
	&\geq \sum_{j=1}^k \frac{C}{|t_{j+1}-t_j|^2} \int_{t_j}^{t_{j+1}} |\gamma'|^2\ds\\
	&= C \sum_{j=1}^k \frac1{|t_{j+1}-t_j|}.
\end{align*}
The sum on the right becomes minimal for equi-distant points $t_j = \frac{j-1}k\,L$ giving rise to the estimate
\[
\W(\gamma) \geq C\sum_{j=1}^{\com k}\frac{1}{L/k} = \frac{C\,k^2}L.
\]
\end{proof}

Referring the reader to a more classical treatment of the fact that the only closed elasticae are the circle, a figure eight curve and their periodic covers, we prove that only the once covered circle can be approximated by embedded curves.

\begin{proof}[Proof of Lemma \ref{lemma approximable curves}]
The transversal self-crossing of the figure eight is easily excluded when writing the curves locally as graphs and using the intermediate value theorem. The multiply covered circle, which only has tangential self-contact, is slightly harder to exclude.

Any curve $\gamma$ which is $W^{2,2}$- or more generally $C^1$-close to an $m$-fold covered circle can be written as a radial graph
\[
\gamma(s) = r(s) \begin{pmatrix} \cos s\\ \sin s\end{pmatrix}
\]
for a $2\pi m$-periodic function $r$ which is $C^1$-close to the constant $1$-function, applying a general statement about writing a surface as a normal graph over a $C^1$-close surface. If $m>1$, we consider the shifted function $\tilde r(s) = r(s + 2\pi)$ and pick an interval $[a,b]\subset [0,2\pi m]$ such that $r(a) = \max r$, $r(b) = \min r$. By the intermediate value theorem
\[
\tilde r(a) - r(a) = r(a+2\pi) - \max r\leq 0, \qquad \tilde r(b) - r(b) = \tilde r(b) -\min r\geq 0
\]
imply that there exists $ s\in [a,b]$ such that $\tilde r(s) - r(s) = 0$
since $r, \tilde r$ are continuous. Then 
\[
\gamma(s+2\pi) = r(s+2\pi) \begin{pmatrix} \cos(s+2\pi)\\ \sin(s+2\pi)\end{pmatrix} = \tilde r(s) \begin{pmatrix}\cos s\\ \sin s\end{pmatrix} = r(s) \begin{pmatrix}\cos s\\ \sin s\end{pmatrix}  = \gamma(s),
\]
which means that $\gamma$ is not embedded. 
\end{proof}

\begin{remark}
While we chose an elementary argument, there are more powerful tools that would cover larger classes of curves. Assuming that a curve $\gamma$ of length $L$ is parametrised by unit speed, we can write $\gamma'(s) = \big(\cos \omega(s), \sin \omega(s)\big)$ and compute that the curvature of $\gamma$ is $\kappa = \omega'(s)$. It follows that 
\[
\int_\gamma \kappa\d\H^1 = \int_0^{L}\omega'(s)\ds = \omega(L) - \omega(0) \in 2\pi \Z
\]
since $\gamma'(0) = \gamma'(L)$. The function $\omega$ is called a Gauss representation of $\gamma$. The quantity $\omega(L) - \omega(0)\in 2\pi\Z$ which measures how often the tangent turns is the Whitney index of the curve.

It is clear that for the circle and all curves $C^1$-close to a circle, we have $\int_\gamma\kappa\d\H^1 = 2\pi$. Since the space of embedded curves is connected (every embedded curve becomes a round circle under curve shortening flow), all embedded curves must have Whitney index $2\pi$, while an $m$-fold covered circle has Whitney index $2\pi m$ and any figure eight curve has Whitney index $0$. 

The same result on the Whitney index of an embedded curve can be obtained by using the Gauss-Bonnet theorem on the disk bounded by the curve $\gamma$ due to Jordan's curve theorem instead of the connectedness of the space of embedded curves. 
\end{remark}


\end{document}